\documentclass[11pt]{amsart}
\usepackage[dvipsnames]{xcolor}
\usepackage{amsfonts,amssymb,amsmath,amscd,amstext}
\usepackage{mathtools}
\usepackage[utf8]{inputenc}
\usepackage{graphicx}
\usepackage{mathrsfs}
\usepackage{changes}
\usepackage{comment}
\usepackage{mathdots}
\usepackage{mathabx}
\usepackage{aligned-overset}
\usepackage[a4paper,top=2cm,bottom=2cm,left=1.7cm,right=1.7cm]{geometry}
\newcommand{\hhh}{{\mathcal{H}}}
\newcommand{\rr}{\mathbb R}
\newcommand{\hh}{{\mathbb{H}}} 
\newcommand{\VfH}[1]{{\mathbb{X}(\mathbb{H}^{#1})}} 
\newcommand{\bb}{{\bf {b}}} 
\newcommand{\JJ}{{\bf {J}}} 
\renewcommand{\rho}{\varrho} 
\newcommand{\Om}{\Omega}
\newcommand{\eps}{\varepsilon}
\DeclarePairedDelimiter{\abs}{\lvert}{\rvert}
\DeclarePairedDelimiter{\norm}{\lVert}{\rVert}
\DeclareMathOperator{\divv}{div}
\DeclareMathOperator{\dsym}{D^{\mathrm sym}}
\DeclareMathOperator{\spann}{span}
\usepackage[colorlinks=true,linkcolor=teal,citecolor=purple]{hyperref}
\usepackage[noabbrev,capitalize,nameinlink]{cleveref}
\theoremstyle{plain}
\newtheorem{theorem}{Theorem}[section]
\newtheorem{proposition}[theorem]{Proposition}
\newtheorem{lemma}[theorem]{Lemma}

\theoremstyle{definition}
\newtheorem{definition}[theorem]{Definition}

\theoremstyle{remark}
\newtheorem{remark}[theorem]{Remark}
\crefname{theorem}{Theorem}{Theorems}
\crefname{lemma}{Lemma}{Lemmas}
\crefname{proposition}{Proposition}{Propositions}
\crefname{corollary}{Corollary}{Corollaries}
\crefname{definition}{Definition}{Definitions}
\crefname{remark}{Remark}{Remarks}
\crefname{example}{Example}{Examples}
\title[Renormalization of contact vector fields in $\hh^n$]{Renormalization of contact vector fields with horizontal Sobolev regularity in Heisenberg groups}
\author[L.~Ambrosio]{Luigi Ambrosio}
\address[L.~Ambrosio]{Scuola Normale Superiore, Piazza dei Cavalieri 7, 56126 Pisa (PI), Italy.}
\email{\href{mailto:luigi.ambrosio@sns.it}{luigi.ambrosio@sns.it}}
\author[G.~Somma]{Gianluca Somma}
\address[G.~Somma]{Dipartimento di Matematica ``Tullio Levi-Civita'', Università degli Studi di Padova, via Trieste 63, 35131 Padova (PD), Italy}
\email{\href{mailto:gianluca.somma@phd.unipd.it}{gianluca.somma@phd.unipd.it}}
\author[S.~Verzellesi]{Simone Verzellesi}
\address[S.~Verzellesi]{Dipartimento di Matematica ``Tullio Levi-Civita'', Università degli Studi di Padova, via Trieste 63, 35131 Padova (PD), Italy}
\email{\href{mailto:simone.verzellesi@unipd.it}{simone.verzellesi@unipd.it}}
\author[D.~Vittone]{Davide Vittone}
\address[D.~Vittone]{Dipartimento di Matematica ``Tullio Levi-Civita'', Università degli Studi di Padova, via Trieste 63, 35131 Padova (PD), Italy}
\email{\href{mailto:davide.vittone@unipd.it}{davide.vittone@unipd.it}}
\date{\today}

\thanks{\textit{Memberships and funding information.} The first author has been supported by the MIUR-PRIN 202244A7YL project ``Gradient Flows and Non-Smooth Geometric Structures with Applications to Optimization and Machine Learning''. All other authors are members of the Istituto Nazionale di Alta Matematica (INdAM), Gruppo Nazionale per l'Analisi Matematica, la Probabilità e le loro Applicazioni (GNAMPA).
S.~Verzellesi received funding through INdAM-GNAMPA 2025 Project \emph{Structure of sub-Riemannian hypersurfaces in Heisenberg groups}, CUP ES324001950001. G.~Somma, S.~Verzellesi and D.~Vittone are supported by the University of Padova and received funding through INdAM-GNAMPA 2026 Project \emph{Variational, Geometric, and Analytic Perspectives on Regularity}, CUP E53C25002010001. D.~Vittone is also supported by INdAM project {\em VAC\&GMT}.
}
\numberwithin{equation}{section}
\begin{document}
\begin{abstract} In this paper we obtain the well-posedness of the transport and continuity equations in the Heisenberg groups $\hh^n$ for a class of contact vector fields $\bb$, under natural assumptions on the regularity of $\bb$ not covered by the, now classical, Euclidean theory \cite{MR1022305}.
It is the first example of well-posedness in a genuine sub-Riemannian setting, that we obtain adapting to the $\hh^n$ geometry the mollification strategy of \cite{MR1022305}. In the final part of the paper we illustrate why our result is not covered by
the Euclidean $BV$ case solved by the first author in \cite{MR2096794}, and we compare it with the strategy of \cite{MR3265963}, based on the representation of the commutator by interpolation \`a la Bakry-\'Emery and an integral representation of the symmetrized derivative of $\bb$. 
\end{abstract}
\maketitle

\section{Introduction}

The aim of this paper is the investigation of the well-posedness of the transport and continuity equations (induced by a
possibly time-dependent vector field $\bb(\tau,p)$), respectively
\begin{equation}\label{CE_and_TE}
\displaystyle{\frac{\partial u}{\partial \tau}} - \left\langle \bb,\nabla u \right\rangle + cu = 0,\qquad 
\displaystyle{\frac{\partial u}{\partial \tau}} + \divv( \bb u)= 0
\qquad \text{in $(0,\bar{\tau}) \times \hh^n$,}
\end{equation}
when a genuinely non-Euclidean state space is involved, namely the
Heisenberg group $\hh^n$. In the Euclidean setting this topic has
received a lot of attention (see e.g.~the Lecture Notes \cite{AC_UMI,MR3283066}), starting from the seminal papers
\cite{MR1022305} and \cite{MR2096794}, which focused on Sobolev and $BV$ vector fields respectively. As explained more in detail in \Cref{sec:consequences}, one of the reasons of this attention is the Lagrangian part of the theory: having in mind the classical theory of characteristics for first order PDE's, well-posedness at the
Eulerian (PDE) level can be transferred to well-posedness at the Lagrangian (ODE) level, and conversely. Indeed, this problem has
been attacked by a variety of techniques: here we focus on the Eulerian ones, via well-posedness of \eqref{CE_and_TE}, quoting only  remarkable recent contributions on the Lagrangian side \cite{Bia_Bo_2020,BCD_2021,BCK}. As we discuss more in detail below, working with the canonical coordinates of $\hh^n$, our results can also
be read as new well-posedness results for a class of vector fields in the Euclidean space $\rr^{2n+1}$,
but we keep our focus on the geometric and intrinsic description of $\hh^n$.

\medskip
Based on the strategy introduced in \cite{MR1022305}, a mollification with respect to the spatial variables of the transport equation produces the appearance of a {\it commutator} $\mathscr C_\eps$, due to the fact that
mollification and divergence do not commute:
\begin{equation}\label{commutators_intro} 
\partial_\tau u_\eps - \left\langle \bb,\nabla u_\eps\right\rangle + cu_\eps =\mathscr C_\eps  \qquad \text{in $(-\infty,\bar{\tau}) \times\hh^n$.}
\end{equation}
The main ingredient of the well-posedness of \eqref{CE_and_TE}, linked to various degrees of regularity
of the components of $\bb$ and its very structure, is the proof of {\it strong} $L^1_{{\rm loc}}$ vanishing of the commutators as $\varepsilon \searrow 0$. When this is the case, distributional solutions to \eqref{CE_and_TE} are stable with respect to composition with suitable test functions, i.e.~are \emph{renormalizable} in the sense of \cite{MR1022305}. In turn, under appropriate integrability conditions on $\bb$, the renormalization property grants the well-posedness of 
\eqref{CE_and_TE}. In the present paper, we are interested in (possibly time-dependent) velocity fields with \emph{horizontal} Sobolev regularity. Differently from \cite{MR1022305}, where convergence of the commutators follows essentially by a characterization of Sobolev functions in terms of their first-order difference quotients, in our context we have to face one more 
non-commutativity aspect, due to the very geometry of $\hh^n$. Indeed, the vector fields $Z_1,\ldots,Z_{2n+1}=X_1,\ldots,X_n,Y_1,\ldots,Y_n,T$ representing the natural basis of the Lie algebra of $\hh^n$ have non-trivial commuting relations. Nevertheless, this lack of commutativity can be balanced when $\bb(\tau,\cdot)$ is a {\em contact} vector field and, in particular, it possesses the  structure  in~\eqref{intro_contact_vector_field} below.

\medskip
Contact vector fields in Heisenberg groups and, more generally, in contact manifolds have been extensively studied, see e.g.~\cite{MR788413,MR1317384,Libermann}. Contact vector fields appear naturally in the theory: in fact, for a (smooth) vector field $\bb$ in $\hh^n$ the following conditions are equivalent:
\begin{itemize}
    \item[(i)]  the  flow associated to $\bb$ is Lipschitz continuous with respect to the \emph{Carnot-Carathéodory metric} on $\hh^n$;
    \item[(ii)] the differential of the associated flow preserves the \emph{horizontal distribution}, i.e., the subbundle linearly generated by $ X_1,\ldots,X_n,Y_1,\ldots Y_n$;
    \item[(iii)] the commutator (Lie bracket) between $\bb$ and any horizontal vector field is horizontal;
    \item[(iv)] there exists for $\bb$ a ``generating'' function $\psi$ on $\hh^n$, i.e., a function such that
    \begin{equation} \label{intro_contact_vector_field}
\bb = \sum_{j=1}^n \left(Y_j \psi X_j - X_j \psi Y_j\right)-4\psi T.
\end{equation}    
\end{itemize}
It is the characterization in (iv) that is relevant for our purposes: in fact, we are able to prove the well-posedness of \eqref{CE_and_TE} for vector fields $\bb(\tau,\cdot)$ possessing the particular structure~\eqref{intro_contact_vector_field}  and 
where the generating function $\psi(\tau,\cdot)$ lies in the second-order horizontal Sobolev space $W^{2,s}_\hhh(\hh^n)$. 
When we try to read in Euclidean terms such a regularity for $\bb(\tau,\cdot)$, we see that we are asking less regularity with respect to the horizontal components, as $X_j\psi$ and $Y_j\psi$ are only in the horizontal Sobolev space
$W^{1,s}_\hhh(\hh^n)$, while asking more only on the vertical component, proportional to $\psi$.
Notice that vector fields with different degrees of regularity with respect to different components are not new in the well-posedness literature for \eqref{CE_and_TE}. We mention, for instance, those in \cite{LeBris_Lions_2004} of the form $(\bb_1(x_1),\bb_2(x_1,x_2))$, with $\bb_2$
only measurable with respect to $x_1$, but see also \cite{Champagnat_Jabin_2010,MR3569235}. However, in these papers, the 
commutation of the derivatives with respect to $x_1$ and $x_2$ plays a role in the proof of the strong convergence of suitable commutators. 

\medskip

As a matter of fact, the $L^1_{\rm loc}$-convergence of $\mathscr C_\eps$ seems to rely on a genuinely structural phenomenon. Indeed, the specific form \eqref{intro_contact_vector_field} of contact vector fields yields certain fine cancelations needed in the proof. Accordingly, the multiplicative factor of the vertical component is not a mere conventional choice, but significantly falls within this scheme. This highlights the distinguished role of contact vector fields, as similar cancellations seem not to be available for more general classes, neither natural (e.g.,~$\bb$ being horizontal) nor more exotic (e.g.,~$\bb$ being in the linear span of the {\em right}-invariant horizontal subbundle).


\medskip
Our approach is structured as follows. \Cref{sec:prem} provides all calculus tools in $\hh^n$ needed in this paper, including horizontal Sobolev spaces and the presentation of contact vector fields $\bb(\tau,\cdot)$ as in \eqref{intro_contact_vector_field}. Even though some computations are done in coordinates, we adopt an intrinsic notation for most of the objects involved, also bearing in mind potential extensions of our results to more general classes of sub-Riemannian spaces.

\medskip
\Cref{sec:diff_quo} provides the main analytic results regarding convergence of difference quotients in first and second-order horizontal Sobolev spaces. Indeed, while first-order difference quotients suffice to deal with velocity fields with Euclidean Sobolev regularity, the non-commutativity structure of $\hh^n$ generates second-order remainders. In turn, the contact structure of $\bb$ allows to handle them via convergence of second-order difference quotients of the vertical component of $\bb$. We address the convergence of difference quotients by a separate study of the horizontal and vertical contributions. Precisely, denote by
$w_\hhh$ and $w_{2n+1}$ the \emph{horizontal} and \emph{vertical} component of $w=(w_\hhh,w_{2n+1})$. We are interested in difference quotients with respect to the intrinsic structure of $\hh^n$. Accordingly, $\cdot$ denotes its group law, while $\left(\delta_\eps\right)_{\eps>0}$ denotes its natural, anisotropic, homogeneous structure. If $f\in W^{1,s}_\hhh$, we show that 
\begin{equation*}
    \frac{f(p \cdot \delta_\varepsilon(w_\hhh,0)) - 
f(p)}{\varepsilon}\xrightarrow[\eps\searrow 0]{L^s}\left\langle\nabla^\hhh f(p),w_\hhh\right\rangle,\qquad \frac{f(p \cdot \delta_\varepsilon(0,w_{2n+1})) - 
f(p )}{\varepsilon}\xrightarrow[\eps\searrow 0]{L^s} 0.
\end{equation*}
Instead,
if $f\in W^{2,s}_\hhh$, we show that
\begin{equation*}
    \frac{f(p \cdot \delta_\varepsilon(w_\hhh,0)) - 
f(p)-\eps\left\langle\nabla^\hhh f(p),w_\hhh\right\rangle}{\varepsilon^2}\xrightarrow[\eps\searrow 0]{L^s} \frac{1}{2} \left\langle \nabla^{2,\hhh} f(p) w_\hhh, w_\hhh \right\rangle
\end{equation*}
and 
\begin{equation*}
    \frac{f(p \cdot \delta_\varepsilon(0,w_{2n+1})) - 
f(p )}{\varepsilon^2}\xrightarrow[\eps\searrow 0]{L^s} w_{2n+1} Tf(p).
\end{equation*}
With a perturbative argument, these results can be used to obtain a crucial Taylor expansion
(still in $L^s$-sense) in \Cref{thm_convergence_2nd_diff_quot}:
\begin{equation}\label{intro_taylorexp}
        \frac{f(p \cdot \delta_\varepsilon(w)) - 
f(p)-\eps\left\langle\nabla^\hhh f(p),w_\hhh\right\rangle}{\varepsilon^2}\xrightarrow[\eps\searrow 0]{L^s} \frac{1}{2} \left\langle \nabla^{2,\hhh} f(p) w_\hhh, w_\hhh \right\rangle+w_{2n+1} Tf(p).
\end{equation}
%

\medskip
\Cref{sec:commu} uses this expansion to prove the strong $L^1_{\rm loc}$-convergence of commutators
$\mathscr C_\epsilon$ in \eqref{commutators_intro}; we emphasize once again that this computation uses in an essential way the cancellations provided by the special structure of $\bb$ in \eqref{intro_contact_vector_field}. 
More precisely, $\mathscr C_\epsilon$ admits an integral representation which splits into two components: one depending on the reaction term $c$, the other depending on the velocity field $\bb$ (cf. \Cref{prop_formadelresto}). If the former can be managed by a standard procedure, the latter needs to be handled by a careful combination of the special structure of $\bb$ provided by \eqref{intro_contact_vector_field} and the Taylor expansion \eqref{intro_taylorexp}. Denote by $b_1,\ldots,b_{2n}$ the horizontal components of $\bb$, and by $b_{2n+1}$ its vertical component. Beside the first-order, horizontal contribution 
\begin{equation}\label{intro_contributoorizz}
    \sum_{j=1}^{2n} \int_{\hh^n} \frac{b_j(p \cdot \delta_\varepsilon(w))-b_j(p)}{\varepsilon} u(p \cdot \delta_\varepsilon(w)) Z_j \rho(w) \,dw
\end{equation}
reminiscent of its Euclidean counterpart, the contact structure of $\bb$ crucially intervenes to put in evidence the second-order, vertical correction 
\begin{equation}\label{intro_contributovert}
    \int_{\hh^n} \frac{b_{2n+1}(p \cdot \delta_\varepsilon(w))-b_{2n+1}(p) - \varepsilon \left\langle \nabla^\hhh b_{2n+1}(p), w_\hhh \right\rangle}{\varepsilon^2} u(p \cdot \delta_\varepsilon(w)) T \rho(w) \,dw,
\end{equation}
because it grants the vanishing of the the remainder 
\begin{equation*}
    \int_{\hh^n} \frac{\left\langle \nabla^\hhh b_{2n+1}(p) + 4\JJ(\bb(p)), \delta_\varepsilon(w) \right\rangle}{\varepsilon^2} u(p \cdot \delta_\varepsilon(w)) T \rho(w) \,dw.
\end{equation*}
In the above formula, $\JJ$ is a suitable ninety-degrees rotation also known as \emph{complex structure} (see \eqref{eq_J}). Convergence of \eqref{intro_contributoorizz} and \eqref{intro_contributovert} is then ensured by \eqref{intro_taylorexp}.

\medskip
Once convergence of commutators has been ensured, the consequences at the level of well-posedness of the transport equation are obtained following 
the path of \cite{MR1022305}, with essentially no modification, under natural growth conditions on $u$, $\bb$ and its spatial divergence,
see \Cref{teo_distributional_implies_renormalized} and \Cref{teo_uniqueness_null_solution}.

\medskip
In \Cref{sec:consequences} we derive for the sake of illustration some consequences of the renormalization property, namely the existence of renormalized solutions to the continuity equation for arbitrary measurable initial data (based on \cite{MR1022305}) and the existence and the uniqueness of Regular Lagrangian Flows in $\hh^n$, see \Cref{def:RLF} (based on \cite{MR2096794}, under one-sided $L^\infty$ bounds on the spatial divergence of $\bb$).

\medskip
\Cref{sec:connection} provides, on the basis of the functional analytic argument of \Cref{Baire}
and the use of highly oscillating test functions, the proof of the density in $W^{2,s}_\hhh(\hh^n)$ of generators of contact vector fields whose horizontal components do not belong to $BV(\rr^{2n+1})$. It follows that our result cannot be deduced from \cite{MR1022305} or \cite{MR2096794}. We believe that similar arguments can be used to rule out the use of other regularity classes of Euclidean vector fields considered in the literature. In the end of the section we compare our approach with that of \cite{MR3265963}, where well-posedness results have been obtained in a large class of metric measure spaces $(X,d,\mu)$, including Riemannian manifolds and the $RCD(K,\infty)$ metric measure spaces of \cite{MR3205729}. It seems that, while the setting of \cite{MR3265963} rules out non-horizontal vector fields, the estimate on the symmetrized derivative $\dsym\bb$ (see \eqref{definizioonedsym}) introduced in \cite{MR3265963} requires a non-horizontal structure. We believe that the conciliation of these two techniques requires further investigations and we leave them for the future.

\section{Preliminaries}\label{sec:prem}
\subsection{Main notation}
We write $\infty=+\infty$. If $\bar{\tau} \in (0,\infty]$, the notation $[0,\bar{\tau}]$ stays for the usual closed interval if $\bar{\tau}<\infty$ or $[0,\infty)$ if $\bar{\tau}=\infty$. Given open sets $A,\Om$, when $\overline{A}$ is a compact subset of $\Om$, we write $A\Subset\Om$. 
We denote by $\langle\cdot,\cdot\rangle_{\rr^N}$ and by $|\cdot|_{\rr^N}$ the Euclidean scalar product and its induced norm. If $s \in [1,\infty]$, we denote its H\"older conjugate by $s'\in [1,\infty]$, i.e., $\frac{1}{s}+\frac{1}{s'}=1$.
Given a real-valued function $f$, we denote its support by $\mathrm{supp}(f)$.
If $f$ is differentiable at a point $q$ in its domain, we denote by $df_q$ its differential at $q$. We denote either by $\mathcal{L}^N$ or by $|\cdot|$ the $N$-dimensional Lebesgue measure. If $\mu$ is a measure and $f$ is a $\mu$-measurable function, we denote the associated push-forward measure by $f_\#\mu$.

\subsection{Heisenberg groups}
Fix $n \in \mathbb{N}\setminus\{0\}$, and set $N\coloneqq 2n+1$. We endow $\rr^N$ with the (non-abelian) group law
\begin{equation*}
\begin{pmatrix}
x \\ y \\ t
\end{pmatrix} \cdot \begin{pmatrix}
x' \\ y' \\ t'
\end{pmatrix} = \begin{pmatrix}
x+x' \\ y+y' \\ t + t' +2 \sum_{j=1}^n(x_j'y_j - x_j y'_j)
\end{pmatrix}
\text{ \qquad for every $(x,y,t),(x',y',t') \in \rr^n \times \rr^n \times \rr$.}
\end{equation*}
In this way, $(\rr^N,\cdot)$ is (the most relevant instance of) a \emph{stratified Lie group} (see \cite{MR2363343}), known as \emph{$n$-th Heisenberg group} and denoted from now on by $\hh^n$. Observe that the identity element is $0$ and the inverse of $p \in \hh^n$ is $p^{-1}=-p$. Given $w=(w_1,\ldots,w_N)=(x,y,t) \in\hh^n$, we may adopt the notation
\begin{equation*}
w_\hhh=(w_1,\ldots,w_{2n})=(x,y) \in \rr^n \times \rr^n,\qquad|w_\hhh|=\sqrt{w_1^2+\ldots+w_{2n}^2}.
\end{equation*}
We may also write
\begin{equation*}
\begin{split}
    x_j(w)=w_j, \qquad y_j(w)=w_{j+n} \qquad \text{for every $j=1,\ldots,2n$. }
\end{split}
\end{equation*}
The left and right translations are, respectively, defined by
\begin{equation*}
L_q(p) = q \cdot p, \qquad R_q(p) = p \cdot q \qquad \text{for every $p,\,q \in \hh^n$.}
\end{equation*}
A basis of the Lie algebra of left-invariant vector fields is given by
\begin{equation*}
\begin{split}
X_j (p) & = \left(dL_{p}\right)_0 \left(\partial_{x_j}\right)=\partial_{x_j}+2y_j \partial_t,\\
Y_j (p) & = \left(dL_{p}\right)_0 \left(\partial_{y_j}\right)=\partial_{y_j}-2x_j \partial_t,\\
T (p) & = \left(dL_{p}\right)_0 \left(\partial_t\right)=\partial_t
\end{split}
\qquad \text{for every $j=1,\ldots,n$, $p=(x,y,t) \in\hh^n$.}
\end{equation*}
Observe that the only nonzero commutators are
\begin{equation} \label{eq_commutation}
[Y_j,X_j]=4T \qquad \text{for every $j=1,\ldots,n$}.
\end{equation}
Because of \eqref{eq_commutation}, the \emph{horizontal distribution} $\hhh$, defined by
\begin{equation*}
    \hhh_p=\spann(X_1(p),\ldots,X_n(p),Y_1(p),\ldots,Y_n(p))\qquad\text{for every $p\in\hh^n$,}
\end{equation*}
satisfies the so-called \emph{H\"ormander condition} (see ~\cite{MR2363343}). 
We endow $\hh^n$ with the unique left-invariant Riemannian metric $\langle \cdot,\cdot \rangle$ that makes this basis orthonormal. We may also write
\begin{equation*}
Z_j = X_j, \qquad Z_{j+n} = Y_j, \qquad Z_N = T \qquad \text{for every $j=1,\ldots,n$.}
\end{equation*}
Denoting by $\VfH{n}$ the class of locally integrable vector fields in $\hh^n$,
any $\bb\in\VfH{n}$ can be canonically written as $\sum_{j=1}^N b_jZ_j$ and 
$\langle\bb,\bb'\rangle=\sum_{j=1}^N b_jb_j'$. Accordingly, with respect to this basis, the intrinsic gradient of a function $f$ (with suitable regularity) is denoted by
\begin{equation*}
\nabla f = (X_1 f,\ldots,X_n f,Y_1 f,\ldots,Y_n f,Tf),
\end{equation*}
while occasionally we denote by $\nabla^{\rr^N}f$ the gradient with respect to the Euclidean structure. The Riemannian measure induced by $\langle\cdot,\cdot\rangle$ coincides with the Lebesgue measure $\mathcal L^N$. Accordingly, the divergence of a vector field $\bb$, defined by the formula
\begin{equation*}
    \int_{\hh^n}d\varphi(\bb)\,dp=-\int_{\hh^n}\varphi\,\divv\bb\,dp\qquad\text{for every $\varphi\in C^\infty_c(\hh^n)$},
\end{equation*}
coincides with the Euclidean divergence. Moreover,
\begin{equation*}
    \divv \left(\sum_{j=1}^Nb_j Z_j\right)=\sum_{j=1}^NZ_jb_j.
\end{equation*}
In addition, if $u$ is a smooth function and $\bb\in\VfH{n}$, one has
\begin{equation}\label{eq:consistency_scalar}
    \left\langle \nabla^{\rr^N}u,\bb\right\rangle_{\rr^N}=du(\bb)=\left\langle\nabla u,\bb\right\rangle.
\end{equation}
By the above considerations, the transport equation can be equivalently formulated owing to the Riemannian structure that we have fixed.
The \emph{horizontal gradient} and the \emph{horizontal Hessian} (see \cite{MR2363343}) are, respectively, defined by
\begin{equation*}
\begin{split}
    \nabla^\hhh f & = (X_1f,\ldots,X_n f,Y_1 f,\ldots,Y_n f),\\
    \left(\nabla^{2,\hhh} f\right)_{ij} & = Z_i Z_j f \qquad \text{for every $i,j=1,\ldots,2n$.}
\end{split}
\end{equation*}
We will also consider the right-invariant vector fields
\begin{equation} \label{eq_right_vector_fields}
\begin{split}
X_j^r (p) & = Z_j^r(p)=\left(dR_{p}\right)_0 \left(\partial_{x_j}\right)=X_j(p)-4y_jT(p),\\
Y_j^r (p) & = Z_{n+j}^r(p)=\left(dR_{p}\right)_0 \left(\partial_{y_j}\right)=Y_j(p)+4x_jT(p),\\
T^r (p) & = Z_{N}^r(p)=\left(dR_{p}\right)_0 \left(\partial_t\right)=T(p)
\end{split}
\qquad \text{for $j=1,\ldots,n$, $p=(x,y,t) \in \hh^n$.}
\end{equation}
Left-invariant vector fields are complete (see \cite[Theorem 9.18]{MR2954043}). Accordingly, if $Z$ is a left-invariant vector field and $\gamma$ is its integral curve starting from $0$, we set
\begin{equation*}
\exp(Z) = \gamma(1).
\end{equation*}
We endow $\hh^n$ with a \emph{homogeneous structure} provided by \emph{intrinsic dilations} (see ~\cite{MR2363343}). 
Namely, we set 
\begin{equation*}
\delta_\lambda(x,y,t) = (\lambda x,\lambda y,\lambda^2 t) \qquad \text{for every $\lambda \geq 0$, $(x,y,t) \in\hh^n$}.
\end{equation*}
In this way, $\delta_\lambda$ is a Lie group isomorphism of $\hh^n$ for any $\lambda>0$. Moreover, we will exploit the \emph{complex structure} 
$\JJ: \hh^n \to \hh^n$ (cf.~\cite{MR2165405}) defined by
\begin{equation} \label{eq_J}
\JJ(x,y,t)=(-y,x,0) \qquad \text{for every $(x,y,t) \in \hh^n$}.
\end{equation}
 It is easy to see that
\begin{equation} \label{eq_J_properties}
\begin{split}
  \JJ(\JJ(w)) & =(-w_\hhh,0),\\
  \langle \JJ(w), z\rangle &=-\langle w, \JJ(z)\rangle
\end{split}
\qquad \text{for every $w,z \in\hh^n$.}
\end{equation}
We equip $\hh^n$ with the so-called \emph{Carnot-Carathéodory distance} $d$ (see \cite{MR3587666}). We recall that, given $p,q\in\hh^n$, 
\begin{equation*}
    d(p,q)=\inf\left\{\int_0^1\sqrt{\left\langle\dot\gamma,\dot\gamma\right\rangle}\,d\tau\,:\, \gamma:[0,1]\to\hh^n\text{ is absolutely continuous, horizontal, } \gamma(0)=p,\,\gamma(1)=q\right\},
\end{equation*}
where an absolutely continuous curve is \emph{horizontal} if $\dot\gamma(\tau)\in\hhh_{\gamma(\tau)}$ for a.e.~$\tau$.
The Carnot-Carathéodory distance is compatible both with the group structure and the homogeneous structure, namely
\begin{equation} \label{eq_CC_distance}
\begin{split}
    d(w \cdot p, w \cdot q) & = d(p,q),\\
    d(\delta_\lambda(p),\delta_\lambda(q)) & = \lambda d(p,q)
\end{split}
\qquad \text{for every $p,q,w \in \hh^n,\,\lambda>0$.}
\end{equation}
We indicate by
\begin{equation*}
B(p,r)=\{q \in\hh^n : d(p,q)<r\}
\end{equation*}
the open ball with center $p \in \hh^n$ and radius $r>0$. We denote by $\mathrm{Lip} (f)$ the Lipschitz constant of $f:\hh^n\to\rr$ with respect to $d$.
The Lebesgue measure $|\cdot|_{\hh^n}$ is both a left and a right Haar measure of $(\hh^n,\cdot)$ (see \cite{MR3587666}). In particular,
\begin{equation} \label{eq_Lebesgue_Haar}
\begin{split}
    \abs*{L_p(E)}&=\abs{E},\\
    \abs*{R_p(E)}&=\abs{E}
\end{split}
\qquad \text{for every $E \subseteq\hh^n$ measurable, $p \in\hh^n$.}
\end{equation}
Moreover,
\begin{equation} \label{eq_dilation_Lebesgue}
\abs*{\delta_\lambda(E)}=\lambda^Q\abs*{E} \qquad \text{for every $E \subseteq\hh^n$ measurable, $\lambda \geq 0$,}
\end{equation}
where $Q\coloneqq 2n+2$ is known as \emph{homogeneous dimension} of $(\hh^n,\cdot)$ (see \cite{MR3587666}). Left and right-invariant vector fields are related by the inversion map as follows.
\begin{lemma}\label{lemscambiarederivateeinversione}
Define $\iota: \hh^n\to\hh^n$ by $\iota(q)=q^{-1}$. Then
\begin{equation} \label{scambiarederivateeinversione}
Z_j(\varphi \circ \iota)=-Z_j^r \varphi \circ \iota \qquad \text{for every } \varphi \in C^1(\hh^n),\, j=1,\ldots,N.
\end{equation}
In particular, if $\varphi=\varphi\circ\iota$, then $Z_j\varphi=-Z_j^r\varphi\circ\iota$.
\end{lemma}
\begin{proof}
For every $q \in \hh^n$, the identity $\iota \circ L_q=R_{q^{-1}} \circ \iota$ holds. Differentiating it at $0$,
\begin{equation*}
d\iota_q((dL_q)_0(Z_j(0))=(dR_{q^{-1}})_0(d\iota_0(Z_j(0)),
\end{equation*}
that is, thanks to the left-invariance of $Z_j$ and the fact that $-d\iota$ is the identity map,
\begin{equation} \label{eq_differential_inversion}
d\iota_q(Z_j(q))=-(dR_{q^{-1}})_0(Z_j(0))=-Z_j^r\left(q^{-1}\right).
\end{equation}
Therefore,
\begin{equation*}
\begin{split}
    Z_j(\varphi \circ \iota)(q)  
    =d\varphi_{q^{-1}}(d\iota_q(Z_j(q)))\overset{\eqref{eq_differential_inversion}}{=}-d\varphi_{q^{-1}}\left(Z_j^r\left(q^{-1}\right)\right) =-Z_j^r\varphi\left(q^{-1}\right) =-(Z_j^r\varphi \circ \iota)(q),
\end{split}
\end{equation*}
which is the thesis.
\end{proof}

We recall that right translations are continuous in $L^s(\hh^n)$. 

\begin{lemma} \label{lem_continuity_translations}
Let $s \in [1,\infty)$ and $f \in L^s(\hh^n)$. Then $\lim_{q\to 0}\norm*{f \circ R_q - f}_{L^s(\hh^n)} =0$.
\end{lemma}
\begin{proof}
The family of linear operators $\{f \circ R_q - f\}_{q\in\hh^n}$ is uniformly bounded in $L^s(\hh^n)$ due to \eqref{eq_Lebesgue_Haar}. Since they converge to $0$ as $q\to 0$ when $f\in C_c(\hh^n)$, the density of $C_c(\hh^n)$ in $L^s(\hh^n)$ grants that the same property holds for any $f\in L^s(\hh^n)$.
\end{proof}
A similar property holds for left translations, with the same proof. In the sequel many other convergence results, in Sobolev classes, will be achieved by a similar functional analytic argument.

\subsection{Group convolution}
For an account on group convolution, we refer to \cite{MR0657581}. Given $f,g\in L^1_{\mathrm{loc}}(\hh^n)$, their \emph{group convolution} $f*g$ is defined by
\begin{equation*}
(f*g)(p) \coloneqq \int_{\hh^n} f(q)g\left(q^{-1} \cdot p\right) \, dq \overset{\eqref{eq_Lebesgue_Haar}}{=} 
\int_{\hh^n} f\left(p \cdot q^{-1}\right)g(q) \, dq \qquad \text{for every $p \in \hh^n$,}
\end{equation*}
provided that the integrals converge. If the domain of $f$ and $g$ is an arbitrary open set $\Omega \subseteq \hh^n$, we extend them to be $0$ outside $\Omega$, and $f*g$ is still well-defined. If in addition $g \in C^1(\hh^n)$, 
\begin{equation} \label{eq_derivatives_convolution}
\begin{split}
    Z(f*g) & =f*Zg \,\qquad \text{if $Z$ is left-invariant,}
\end{split}
\end{equation}
since $f*g$ can be viewed as the superposition, weighted by $f$, of left translations of $g$.
In order to introduce group mollification, we fix a mollifier $\rho \in C_c^\infty(\hh^n)$ such that
\begin{equation}\label{eq_mollificatori}
\rho \geq 0, \quad \rho(p)=\rho\left(p^{-1}\right), \quad \int_{\hh^n} \rho\,dp = 1 \quad \text{and} \quad \textrm{supp}(\rho) \subseteq B(0,1).
\end{equation}
For every $\varepsilon>0$, set
\begin{equation} \label{eq_def_mollifier}
\rho_\varepsilon(p)=\frac{1}{\varepsilon^Q}\rho \left(\delta_{\frac{1}{\varepsilon}}(p) \right) \qquad \text{for every $p \in \hh^n$.}
\end{equation}
Notice that
\begin{equation} \label{eq_derivative_mollification}
\begin{split}
    Z^r_j\rho_\varepsilon(p) &=\frac{1}{\varepsilon^{Q+1}}Z^r_j\rho \left(\delta_{\frac{1}{\varepsilon}}(p) \right) \qquad \text{for every $j=1,\ldots,2n$,\, $p\in\hh^n$,}\\
    T\rho_\varepsilon(p) &=\frac{1}{\varepsilon^{Q+2}}T\rho \left(\delta_{\frac{1}{\varepsilon}}(p) \right)\qquad\text{for every  $p\in\hh^n$,}
\end{split}
\end{equation}
and analogous formulas hold for every $Z_j$. The following proposition collects standard convergence properties of group mollification.
\begin{proposition} \label{prop_group_mollification}
Let $s \in [1,\infty)$. Let $u \in L^s(\hh^n)$ and set $u_\varepsilon=u * \rho_\varepsilon$ for every $\varepsilon>0$. Then
\begin{equation} \label{eq_Young_inequality}
\begin{split}
    \norm*{u*\rho_\varepsilon}_{L^s(\hh^n)}
    & \leq \norm*{u}_{L^s(\hh^n)}
\end{split}
\qquad \text{for every $\varepsilon>0$.}
\end{equation}
Moreover, $u_\varepsilon \xrightarrow[]{L^s} u$ as $\varepsilon \searrow 0$.
\end{proposition}
If $\mathcal F$ is a distribution acting on $C^\infty_c\left([0,\bar \tau)\times\hh^n\right)$ and $\rho\in C^\infty_c(\hh^n)$, we define the \emph{spatial group convolution of $\mathcal F$ by $\rho$} by 
\begin{equation}\label{deficonvdistr}
    \left\langle \mathcal F*\rho,\varphi\right\rangle\coloneqq\left\langle \mathcal F,\varphi*\check\rho\right\rangle\qquad\text{for every $\varphi\in C^\infty_c\left([0,\bar \tau)\times\hh^n\right)$,}
\end{equation}
where $\check{\rho}(p)\coloneqq\rho\left(p^{-1}\right)$, so that
\begin{equation*}
\qquad\left(\varphi*\check\rho \right)(\tau,p)=\int_{\hh^n}\varphi(\tau,q)\check\rho\left(q^{-1}\cdot p\right)\,dq =\int_{\hh^n}\varphi(\tau,q)\rho\left(p^{-1}\cdot q\right)\,dq \qquad\text{for every $(\tau,p)\in[0,\bar\tau)\times\hh^n$.}
\end{equation*}

\subsection{Horizontal Sobolev spaces and contact vector fields}
As pointed out in the introduction, we will consider vector fields $\bb\in\VfH{n}$ with \emph{horizontal} Sobolev regularity. To this end, fix an open subset $\Om\subseteq\hh^n$ and $u\in L^1_{\mathrm{loc}}(\Om)$. If $j=1,\ldots,2n$, we define the distribution $Z_j u$ by 
\begin{equation*}
\left\langle Z_ju,\varphi\right\rangle=-\int_\Om uZ_j\varphi\,dp \qquad \text{for every $\varphi\in C^\infty_c(\Om)$}.
\end{equation*}
If $s\in[1,\infty]$, we define the \emph{horizontal Sobolev space} $W^{1,s}_\hhh(\Om)$ by
\begin{equation*}
    W^{1,s}_\hhh(\Om)\coloneqq\{u\in L^s(\Om)\,:\,Z_j u\in L^s(\Om)\text{ for every $j=1,\ldots,2n$}\}.
\end{equation*}
The spaces $W^{1,s}_{\hhh,\mathrm{loc}}(\Om)$ and $W^{k,s}_\hhh(\Om)$, for $k=2,3,\ldots$, are defined accordingly. It is well-known (see \cite{MR0657581}) that the vector space $W^{k,s}_\hhh(\Om)$, endowed with the norm
\[
\Vert u\Vert_{W^{k,s}_\hhh(\Om)}\coloneqq\Vert u\Vert_{L^s(\Om)}+\sum_{j=1}^{2n}\|Z_j u\|_{L^s(\Om)}+\cdots+\sum_{j_1,\ldots,j_k=1}^{2n}\|Z_{j_1}\cdots Z_{j_k}u\|_{L^s(\Om)},
\]
is a Banach space for every $1\leq s\leq \infty$ and $k\in\mathbb N \setminus \{0\}$, reflexive when $1<s<\infty$.

The following Meyers-Serrin type approximation result holds (cf.~\cite{MR1404326}, and cf.~\cite{MR1437714} when $k=1$).
\begin{theorem} \label{meyersserrin}
    Let $\Om\subseteq \hh^n$ be an open set. Let $k\in\mathbb N\setminus\{0\}$ and $s\in[1,\infty)$. Let $u\in W^{k,s}_\hhh(\Om)$. There exists a sequence $\{u_h\}_{h\in\mathbb N}\subseteq C^\infty(\Om)\cap W^{k,s}_\hhh(\Om)$ such that 
    \begin{equation*}
        \lim_{h\to\infty }\|u_h-u\|_{W^{k,s}_\hhh(\Om)}=0.
    \end{equation*}
\end{theorem}
The commutation relation \eqref{eq_commutation} holds as well in the horizontal Sobolev setting.
\begin{lemma} \label{lem_W21hinW11}
    Let $\Omega\subseteq\hh^n$ be an open subset. Let $s\in[1,\infty]$. Then $W^{2,s}_\hhh(\Omega)\subseteq W^{1,s}(\Omega)$, and 
    \begin{equation*} \label{eq_embedding2sin1s}
        Tu=\frac{1}{4}(Y_jX_j-X_jY_j)u\qquad\text{for every $u\in W^{2,s}_\hhh(\Omega)$, $j=1,\ldots,n$.}
    \end{equation*}
\end{lemma}
\begin{proof}
    Let $\varphi\in C^\infty_c(\Omega)$. Let $j=1,\ldots,n$. It suffices to notice that
    \begin{equation*}
        \begin{split}
            \int_\Omega u\,T\varphi \,dp=\frac{1}{4}\int_\Omega u\left(Y_jX_j\varphi-X_jY_j\varphi\right) \,dp
            =\frac{1}{4}\int_\Omega\left(X_j Y_ju-Y_jX_j u\right)\varphi \,dp=-\int_\Omega\frac{1}{4}(Y_jX_j-X_jY_j)u\,\varphi \,dp.
        \end{split}
    \end{equation*}
The thesis follows.
\end{proof}
We employ horizontal Sobolev spaces to introduce vector fields with horizontal Sobolev regularity. We recall that a smooth vector field $\bb\in\VfH{n}$ is called a \emph{contact vector field} if its flow preserves the horizontal distribution $\hhh$ (see \cite{MR1317384}). In this case, $\bb$ admits an explicit representation in terms of a smooth \emph{generating function} $\psi$, namely
\begin{equation} \label{contact_vector_field}
\bb = -4\psi T + \sum_{j=1}^n (Y_j \psi X_j - X_j \psi Y_j) = -4\psi T - \JJ\left(\nabla^\hhh \psi\right),
\end{equation}
where in the second equality we extended the action of $\JJ$ in \eqref{eq_J} also to gradients. Therefore, it makes sense to define contact vector fields with horizontal Sobolev regularity as those representable by a generating function $\psi\in W^{2,s}_\hhh(\hh^n)$ for some $s\in [1,\infty]$ according to \eqref{contact_vector_field}. In this way, \Cref{lem_W21hinW11} implies that $b_N\in W^{1,s}(\rr^N)$. In particular, owing again to 
\Cref{lem_W21hinW11}, it is easy to verify that 
\begin{equation*}
\divv \bb = -4(n+1)T \psi \in L^{s}(\hh^n). 
\end{equation*}

\section{Convergence of difference quotients}\label{sec:diff_quo}
In this section we study the asymptotic behavior of difference quotients of horizontal Sobolev functions. However, differently e.g.~from \cite{MR1459590}, we consider difference quotients along intrinsically dilated left translations rather than along horizontal flows. This fact requires to carefully deal with suitable vertical remainders, as it will be clearer in a while. These results will be fundamental in order to prove the regularization property of the next section. Along this section, $\Om$ is a fixed open set in $\hh^n$.
We begin by recalling the first and second-order behavior of smooth functions along dilated left translations.
\begin{lemma}\label{deriprimasecondasmoothalongdilations}
Let $f\in C^\infty(\Omega)$. Let $p\in\Omega$, $\tau>0$ and $w \in \hh^n$. Assume that $p\cdot\delta_\tau(w)\in\Om$. Then
\begin{equation}\label{derivataalongdilation}
    \frac{d(f(p\cdot\delta_\tau(w)))}{d\tau}=\left\langle\nabla^\hhh f(p\cdot\delta_\tau(w)),w_\hhh\right\rangle+2\tau w_N Tf(p\cdot\delta_\tau(w))
\end{equation}
and
\begin{equation} \label{derivataalongdilation_2nd}
\begin{split}
    \frac{d^2(f(p\cdot\delta_\tau(w)))}{d\tau^2} & = \left\langle\nabla^{2,\hhh} f(p\cdot\delta_\tau(w)) w_\hhh, w_\hhh\right\rangle+2w_N Tf(p\cdot\delta_\tau(w)) \\ & \quad +4\tau w_N\left\langle\nabla^\hhh Tf(p\cdot\delta_\tau(w)), w_\hhh\right\rangle+4\tau^2w_N^2 TTf(p\cdot\delta_\tau(w)).
\end{split}
\end{equation}
\end{lemma}
\begin{proof}
Observe that, for every $\tau>0$,
\begin{equation*}
\frac{d(\delta_\tau(w))}{d\tau}=\sum_{i=1}^{2n}w_iZ_i(\delta_\tau(w))+2\tau w_N T(\delta_\tau(w)).
\end{equation*}
Hence, setting $\gamma(\tau)=p \cdot \delta_\tau(w)=L_p(\delta_\tau(w))$ for every $\tau>0$, we obtain, by left-invariance,
\begin{equation*}
\begin{split}
    \dot{\gamma}(\tau) & =\left(dL_p\right)_{\delta_\tau(w)}\left(\frac{d(\delta_\tau(w))}{d\tau}\right)\\
    & =\sum_{i=1}^{2n} w_i\left(dL_p\right)_{\delta_\tau(w)}(Z_i(\delta_\tau(w)))+2\tau w_N \left(dL_p\right)_{\delta_\tau(w)}(T(\delta_\tau(w)))\\
    & =\sum_{i=1}^{2n} w_iZ_i(\gamma(\tau))+2\tau w_N T(\gamma(\tau)).
\end{split}
\end{equation*}
Having this, the two formulas follow from a direct computation. First,
\begin{equation*}
\begin{split}
    \frac{d(f(p\cdot\delta_\tau(w)))}{d\tau} & = \sum_{i=1}^{2n} Z_i f(p \cdot \delta_\tau(w)) w_i +2\tau Tf(p \cdot \delta_\tau(w)) w_N\\
    & = \left\langle \nabla^\hhh f(p \cdot \delta_\tau(w)), w_\hhh \right\rangle +2\tau w_N Tf(p \cdot \delta_\tau(w)).
\end{split}
\end{equation*}
Furthermore,
\begin{equation*}
\begin{split}
    \frac{d^2(f(p\cdot\delta_\tau(w)))}{d\tau^2} & = \sum_{i,j=1}^{2n} Z_j Z_i f(p \cdot \delta_\tau(w)) w_i w_j +2\tau \sum_{i=1}^{2n} T Z_i f(p \cdot \delta_\tau(w)) w_i w_N \\ & \quad +2Tf(p \cdot \delta_\tau(w)) w_N +2\tau w_N \sum_{i=1}^{2n} Z_i Tf(p \cdot \delta_\tau(w)) w_i + 4\tau^2 w_N^2 TTf(p \cdot \delta_\tau(w))\\
    & = \left\langle\nabla^{2,\hhh} f(p\cdot\delta_\tau(w)) w_\hhh, w_\hhh\right\rangle+2w_N Tf(p\cdot\delta_\tau(w)) \\ & \quad +4\tau w_N\left\langle\nabla^\hhh Tf(p\cdot\delta_\tau(w)), w_\hhh\right\rangle+4\tau^2w_N^2 TTf(p\cdot\delta_\tau(w)).
\end{split}
\end{equation*}
\end{proof}
The next result shows that vertical difference quotients disappear in the limit.
\begin{lemma} \label{lem_horizontal_vs_vertical}
Let $s \in [1,\infty)$ and let $f \in W_\hhh^{1,s}(\Omega)$. Let $A\subseteq\Om$ be open and such that $\mathrm{dist}(A,\partial\Om)>0$. Then
\begin{equation}\label{eq:priori1}
\lim_{\varepsilon \searrow 0}^{L^s(A)} \frac{f(p \cdot \delta_\varepsilon(w)) - f(p \cdot \delta_\varepsilon(w_\hhh))}{\varepsilon} = 0 \qquad \text{for every $w \in \hh^n$.}
\end{equation}
In addition, if $w\in\hh^n$ and $\varepsilon>0$ satisfy 
\begin{equation}\label{condepsw}
    \varepsilon \left(|w_\hhh|+2\sqrt{|w_N|}\right)<\mathrm{dist}(A,\partial\Om),
\end{equation}
one has
\begin{equation}\label{eq:priori2}
   \left\| \frac{f(p \cdot \delta_\varepsilon(w)) - f(p\cdot\delta_\varepsilon(w_\hhh))}{\varepsilon}\right\|_{L^s(A)}\leq 2\sqrt{|w_N|}\left\|\nabla^\hhh f\right\|_{L^s(\Om)}.
\end{equation}
\end{lemma}
\begin{proof}
It is sufficient to show that both properties hold when $f\in C^\infty(\Om)\cap W^{1,s}_\hhh(\Om)$: indeed, if this holds, then \eqref{eq:priori2} extends to
all $f\in  W_\hhh^{1,s}(\Om)$ thanks to \Cref{meyersserrin} (notice that the argument of the limit on the left hand side of \eqref{eq:priori1} is $L^s$-continuous
when $\varepsilon$ is fixed). Then, the operators in the left hand side result uniformly bounded in $W_\hhh^{1,s}(\Om)$, and from the
density of $C^\infty(\Om)\cap W^{1,s}_\hhh(\Om)$ in $W_\hhh^{1,s}(\Om)$ one obtains the validity of \eqref{eq:priori1}.
Fix $f\in C^\infty(\Om)\cap W^{1,s}_\hhh(\Om)$. Fix $p\in A$. If $w_N=0$, there is nothing to prove. Assume that $w_N>0$. Assume that $\varepsilon>0$ satisfies \eqref{condepsw}. By definition of $d$, $d\left( p,p\cdot\delta_\eps(w_\hhh)\right)\leq\eps|w_\hhh|$, whence $p\cdot\delta_\eps(w_\hhh)\in \Om$ by \eqref{condepsw}. We exhibit a horizontal curve joining $p\cdot\delta_\varepsilon(w_\hhh)$ and $p\cdot \delta_\varepsilon(w)$: start from $p\cdot \delta_\varepsilon(w_\hhh)$, follow the flow of $Y_1$ and then, in order, those of $X_1$, $-Y_1$ and $-X_1$, with time $\tilde{\tau} = \frac{\varepsilon}{2} \sqrt{w_N}$ at each step. By the Baker-Campbell-Hausdorff formula (cf.~\cite{MR746308}),
\begin{equation*}\begin{split}
    \exp(\tilde \tau Y_1)&\cdot\exp (\tilde \tau X_1)\cdot\exp (-\tilde\tau Y_1)\cdot\exp(-\tilde \tau X_1)\\
    &=\exp\left(\tilde \tau Y_1+\tilde \tau X_1+\frac{\tilde \tau^2}{2}[Y_1,X_1]\right)\cdot \exp\left(-\tilde \tau Y_1-\tilde \tau X_1+\frac{\tilde \tau^2}{2}[Y_1,X_1] \right)\\
    &=\exp (\tilde\tau^2[Y_1,X_1])\\
    &=\exp(\varepsilon ^2w_N T).
\end{split}\end{equation*}
The case $w_N<0$ could be handled similarly. In particular, if $q$ is any point along the above curve,
\begin{equation*}
    d(p,q)\leq d(p,p\cdot\delta_\eps(w_\hhh))+d(p\cdot\delta_\eps(w_\hhh),q)\leq\eps\left(|w_\hhh|+2\sqrt{|w_N|}\right),
\end{equation*}
whence $q\in\Om$ by \eqref{condepsw}.
Set
\begin{equation*}\begin{split}
    \bar{w}_0 & = \delta_\varepsilon(w_\hhh),\\
    \bar{w}_1 & =\bar{w}_0 \cdot \exp(\tilde{\tau} Y_1),\\
    \bar{w}_2 & =\bar{w}_1 \cdot \exp(\tilde{\tau} X_1),\\
    \bar{w}_3 & =\bar{w}_2 \cdot \exp(-\tilde{\tau} Y_1),\\
    \bar{w}_4 & =\bar{w}_3 \cdot \exp(-\tilde{\tau} X_1)=\delta_\varepsilon(w).
\end{split}\end{equation*}
Then
\begin{equation*}
\begin{split}
    f(p \cdot \delta_\varepsilon(w))-f(p \cdot \delta_\varepsilon(w_\hhh)) & = [f(p \cdot \bar{w}_4)-f(p \cdot \bar{w}_3)+f(p \cdot \bar{w}_2)-f(p \cdot \bar{w}_1)]\\ & \quad + [f(p \cdot \bar{w}_3)-f(p \cdot \bar{w}_2)+f(p \cdot \bar{w}_1)-f(p \cdot \bar{w}_0)]\\
    & =\int_0^{\tilde{\tau}} [-X_1f(p \cdot \bar{w}_3 \cdot \exp(-\tau X_1))+X_1 f(p \cdot \bar{w}_1 \cdot \exp(\tau X_1))] \,d\tau \\ & \quad + \int_0^{\tilde{\tau}} [-Y_1 f(p \cdot \bar{w}_2 \cdot \exp(-\tau Y_1))+Y_1 f(p \cdot \bar{w}_0 \cdot \exp(\tau Y_1))] \,d\tau.
\end{split}
\end{equation*}
 Computing the $L^s$-norm, Minkowski's integral inequality (cf.~\cite[Corollary B.83]{MR3726909}) implies
\begin{equation} \label{eq_stiaallimite}
\begin{split}
    & \left(\int_{A} \abs*{f(p \cdot \delta_\varepsilon(w))-f(p \cdot \delta_\varepsilon(w_\hhh))}^s \,dp\right)^{\frac{1}{s}}\\
    &\quad \leq \int_0^{\tilde{\tau}} \left(\int_{A} \abs*{X_1 f(p \cdot \bar{w}_3 \cdot \exp(-\tau X_1))-X_1 f(p \cdot \bar{w}_1 \cdot \exp(\tau X_1))}^s \,dp\right)^{\frac{1}{s}} \,d\tau \\ & \qquad + \int_0^{\tilde{\tau}} \left(\int_{A} \abs*{Y_1 f(p \cdot \bar{w}_2 \cdot \exp(-\tau Y_1))-Y_1 f(p \cdot \bar{w}_0 \cdot \exp(\tau Y_1))}^s \,dp\right)^{\frac{1}{s}} \,d\tau.
\end{split}
\end{equation}
Dividing \eqref{eq_stiaallimite} by $\varepsilon$, since $\tilde\tau=O(\varepsilon)$ and $f\in C^\infty(\Om)\cap W^{1,s}_\hhh(\Om)$, \Cref{lem_continuity_translations} implies \eqref{eq:priori1}.
 In order to prove \eqref{eq:priori2}, just use the invariance of $\mathcal{L}^N$ under right translations and the explicit formula of $\tilde\tau$. 
The thesis follows.
\end{proof}

We are ready to prove the main convergence results for first-order difference quotients. 
\begin{theorem} \label{thm_convergence_diff_quot}
Let $s \in [1,\infty)$ and let $\Omega \subseteq\hh^n$ be open. Let $f \in W_\hhh^{1,s}(\Omega)$. Let $A\subseteq\Om$ be open and such that $\mathrm{dist}(A,\partial\Om)>0$. Then
\begin{equation}\label{eq:priori3}
\lim_{\varepsilon \searrow 0}^{L^s(A)} \frac{f(p \cdot \delta_\varepsilon(w)) - f(p)}{\varepsilon} = \left\langle \nabla^\hhh f(p), w_\hhh \right\rangle \qquad \text{for every $w\in\hh^n$.}
\end{equation}
In addition, if $\eps>0$ and $w\in\hh^n$ satisfy \eqref{condepsw}, one has
\begin{equation}\label{eq:priori4}
\left\| \frac{f(p \cdot \delta_\varepsilon(w)) - f(p)}{\varepsilon}\right\|_{L^s(A)}\leq \left (|w_\hhh|+2\sqrt{w_N}\right)\left\|\nabla^\hhh f\right\|_{L^s(\Om)}.
\end{equation}
\end{theorem}
\begin{proof}
As in the previous proof, it suffices to assume $f\in C^\infty(\Om)\cap W^{1,s}_\hhh(\Om)$. If $p\in A$ and $\eps>0$ satisfies \eqref{condepsw}, one has
\begin{equation*}
f(p \cdot \delta_\varepsilon(w_\hhh))-f(p)-\varepsilon \left\langle \nabla^\hhh f(p), w_\hhh \right\rangle \overset{\eqref{derivataalongdilation}}{=} \int_0^\varepsilon \left[\left\langle \nabla^\hhh f(p \cdot \delta_\tau(w_\hhh)), w_\hhh \right\rangle - \left\langle \nabla^\hhh f(p), w_\hhh \right\rangle\right] \,d\tau.
\end{equation*}
Dividing by $\varepsilon$, computing the $L^s$-norm and using Minkowski's integral inequality, we get
\begin{equation}\label{eq:priori5}
\begin{split}
    & \left(\int_{A} \left\vert \frac{f(p \cdot \delta_\varepsilon(w_\hhh))-f(p)}{\varepsilon} - \left\langle \nabla^\hhh f(p), w_\hhh \right\rangle \right\vert^s \,dp\right)^{\frac{1}{s}} \\
    &\quad \leq \frac{1}{\varepsilon} \int_0^\varepsilon \left(\int_{A} \abs*{\left\langle \nabla^\hhh f(p \cdot \delta_\tau(w_\hhh))-\nabla^\hhh f(p), w_\hhh \right\rangle}^s \,dp\right)^{\frac{1}{s}} \,d\tau.
\end{split}
\end{equation}
 By \Cref{lem_continuity_translations}, this clearly implies \eqref{eq:priori3}. Then, \eqref{eq:priori4} follows combining
\eqref{eq:priori2} with
\begin{equation}
\left\| \frac{f(p \cdot \delta_\varepsilon(w_\hhh)) - f(p)}{\varepsilon}\right\|_{L^s(A)}\leq |w_\hhh|\left\|\nabla^\hhh f\right\|_{L^s(\Om)}.
\end{equation}
In turn, this inequality can be obtained from the same argument leading to \eqref{eq:priori5}, starting from
\begin{equation*}
f(p \cdot \delta_\varepsilon(w_\hhh))-f(p)=\int_0^\varepsilon \left\langle \nabla^\hhh f(p \cdot \delta_\tau(w_\hhh)), w_\hhh \right\rangle\,d\tau
\end{equation*}
and using the invariance of $\mathcal{L}^N$ under right translations.
\end{proof}
Next, we deal with second-order difference quotients. 
We first compute the contribution of vertical difference quotients. 
\begin{lemma} \label{lem_horizontal_vs_vertical_2nd}
Let $s \in [1,\infty)$ and let $f \in W_\hhh^{2,s}(\Omega)$. Let $A\subseteq\Om$ be open and such that $\mathrm{dist}(A,\partial\Om)>0$. Then
\begin{equation}\label{eq:priori6}
\lim_{\varepsilon \searrow 0}^{L^s(A)} \frac{f(p \cdot \delta_\varepsilon(w)) - f(p \cdot \delta_\varepsilon(w_\hhh))}{\varepsilon^2} = w_N Tf(p)
 \qquad \text{for every $w \in \hh^n$.}
\end{equation}
In addition, if $\varepsilon>0$ and $w\in\hh^n$ satisfy \eqref{condepsw}, one has
\begin{equation}\label{eq:priori7}
\left\| \frac{f(p \cdot \delta_\varepsilon(w))-f(p \cdot \delta_\varepsilon(w_\hhh))}{\varepsilon^2}\right\|_{L^s(A)}\leq
|w_N|\|Tf\|_{L^s(\Om)}.
\end{equation}
\end{lemma}
\begin{proof}
Again, it suffices to assume $f\in C^\infty(\Om)\cap W^{2,s}_\hhh(\Om)$. We know that
\begin{equation}\label{eq:priori9}
f(p \cdot \delta_\varepsilon(w))-f(p \cdot \delta_\varepsilon(w_\hhh)) = 2w_N \int_0^\varepsilon \tau Tf(p \cdot \delta_\varepsilon(w_\hhh) \cdot \delta_\tau(\exp(w_N T))) \,d\tau.
\end{equation}
Hence,
\begin{equation*}\begin{split}
    \frac{f(p \cdot \delta_\varepsilon(w))-f(p \cdot \delta_\varepsilon(w_\hhh))}{\varepsilon^2} - w_N Tf(p) = \frac{2w_N}{\varepsilon^2} \int_0^\varepsilon \tau [Tf(p \cdot \delta_\varepsilon(w_\hhh) \cdot \delta_\tau(\exp(w_N T)))-Tf(p)] \,d\tau.
\end{split}\end{equation*}
Computing the $L^s$-norm and applying Minkowski's integral inequality, we obtain
\begin{equation}\label{eq:priori8}
\begin{split}
    & \left(\int_{A} \left\vert \frac{f(p \cdot \delta_\varepsilon(w))-f(p \cdot \delta_\varepsilon(w_\hhh))}{\varepsilon^2} - w_N Tf(p)\right\vert^s \,dp\right)^{\frac{1}{s}}\\
    & \quad\leq \frac{2 \abs*{w_N}}{\varepsilon^2} \int_0^\varepsilon \tau \left( \int_{A}\abs*{Tf(p \cdot \delta_\varepsilon(w_\hhh) \cdot \delta_\tau(\exp(w_N T)))-Tf(p)}^s \,dp\right)^{\frac{1}{s}}\,d\tau.
\end{split}
\end{equation}
This and \Cref{lem_continuity_translations} clearly imply \eqref{eq:priori6}. Then, \eqref{eq:priori7} follows
by the same argument leading to \eqref{eq:priori8}, starting from \eqref{eq:priori9}
and using the invariance of $\mathcal{L}^N$ under right translations.
\end{proof}
The next result describes the asymptotic behavior of second-order difference quotients.

\begin{theorem} \label{thm_convergence_2nd_diff_quot}
Let $s \in [1,\infty)$ and let $\Omega \subseteq \hh^n$ be open. Let $f \in W_\hhh^{2,s}(\Omega)$. Let $A\subseteq\Om$ be open and such that $\mathrm{dist}(A,\partial\Om)>0$. Then
\begin{equation}\label{eq:priori10}
\lim_{\varepsilon \searrow 0}^{L^s(A)} \frac{f(p \cdot \delta_\varepsilon(w)) - f(p) - \varepsilon \left\langle \nabla^\hhh f(p), w_\hhh \right\rangle}{\varepsilon^2} = \frac{1}{2} \left\langle \nabla^{2,\hhh} f(p) w_\hhh, w_\hhh \right\rangle + w_N Tf(p) 
\end{equation}
for every $w \in \hh^n$. In addition, if $\varepsilon>0$ and $w\in\hh^n$ satisfy \eqref{condepsw}, one has
\begin{equation}\label{eq:priori11}
   \left\| \frac{f(p \cdot \delta_\varepsilon(w)) - f(p)-\varepsilon \left\langle \nabla^\hhh f(p), w_\hhh \right\rangle}{\varepsilon^2}\right\|_{L^s(A)}\leq|w_N|\|Tf\|_{L^s(\Om)}+|w_\hhh|^2\left\|\nabla^{\hhh,2}f\right\|_{L^s(\Om)}.
\end{equation}
\end{theorem}
\begin{proof} As customary, As we assume $f\in C^\infty(\Om)\cap W^{2,s}_\hhh(\Om)$. Fix $\eps>0$ such that \eqref{condepsw} holds. One has
\begin{equation*}
f(p \cdot \delta_\varepsilon(w_\hhh))- f(p)-\varepsilon \left\langle \nabla^\hhh f(p), w_\hhh \right\rangle \overset{\eqref{derivataalongdilation_2nd}}{=} \int_0^\varepsilon (\varepsilon-\tau) \left\langle \nabla^{2,\hhh} f(p \cdot \delta_\tau(w_\hhh)) w_\hhh, w_\hhh \right\rangle \,d\tau,
\end{equation*}
so that
\begin{equation*}
\begin{split}
    f(p \cdot \delta_\varepsilon(w_\hhh))- & f(p)-\varepsilon \left\langle \nabla^\hhh f(p), w_\hhh \right\rangle - \frac{\varepsilon^2}{2} \left\langle \nabla^{2,\hhh} f(p) w_\hhh, w_\hhh \right\rangle\\
    & = \int_0^\varepsilon (\varepsilon-\tau) \left\langle \left[\nabla^{2,\hhh} f(p \cdot \delta_\tau(w_\hhh))-\nabla^{2,\hhh} f(p)\right] w_\hhh, w_\hhh \right\rangle \,d\tau.
\end{split}
\end{equation*}
Dividing by $\varepsilon^2$ and computing the $L^s$-norm we infer (as in the proof of \Cref{lem_horizontal_vs_vertical})
\begin{equation}\label{eq:priori12}
\begin{split}
    & \left(\int_{A} \left\vert \frac{f(p \cdot \delta_\varepsilon(w_\hhh))- f(p)-\varepsilon \left\langle \nabla^\hhh f(p), w_\hhh \right\rangle}{\varepsilon^2} -\frac{1}{2} \left\langle \nabla^{2,\hhh} f_k(p) w_\hhh, w_\hhh \right\rangle \right\vert^s \,dp\right)^{\frac{1}{s}} \\
    & \quad\leq \frac{1}{\varepsilon^2} \int_0^\varepsilon (\varepsilon-\tau) \left(\int_{A} \abs*{\left\langle \left[\nabla^{2,\hhh} f(p \cdot \delta_\tau(w_\hhh))-\nabla^{2,\hhh} f(p)\right] w_\hhh, w_\hhh \right\rangle}^s \,dp\right)^{\frac{1}{s}} \,d\tau\\
    & \quad\leq \frac{\abs*{w_\hhh}^2}{\varepsilon} \int_0^\varepsilon \left(\int_{A} \abs*{\nabla^{2,\hhh} f(p \cdot \delta_\tau(w_\hhh))-\nabla^{2,\hhh} f(p)}^s \,dp\right)^{\frac{1}{s}} \,d\tau.
\end{split}
\end{equation}
Therefore, \eqref{eq:priori10} follows by \Cref{lem_continuity_translations} and by \Cref{lem_horizontal_vs_vertical_2nd}. In addition \eqref{eq:priori7} and \eqref{eq:priori12}
grant \eqref{eq:priori11}.
\end{proof}
Finally, we address the case $s=\infty$. As in the Euclidean case, the uniform convergence of difference quotients is in general false. However, it is still possible to prove their uniform boundedness. For instance, fix $f\in W^{1,\infty}_{\hhh,\mathrm{loc}}(\rr^N)$ and fix bounded open sets $\Om\subseteq\hh^n$ and $\tilde\Om\Subset\Om$. By \cite{MR1631642}, $f$ is Lipschitz continuous on $\Om$ with respect to $d$. Denoting its Lipschitz constant by $\mathrm{Lip}(f,\Om)$, there exists $c_1=c_1(\tilde\Om)>0$ such that 
\begin{equation}\label{lipboundfirstorder}
    \left|\frac{f\left(p\cdot\delta_\epsilon(w)\right)-f(p)}{\varepsilon}\right|\leq  \mathrm{Lip}(f,\Om)d(0,w)\qquad\text{for every $p\in\tilde\Om$, $w\in B(0,1)$, $0<\varepsilon<c_1$.}
\end{equation}
Similar bounds hold for second-order difference quotients. For instance, in the following we may exploit the following result.
\begin{theorem} \label{thm_convergence_2nd_diff_quot_infty}
 Let $f \in W_{\hhh,\mathrm{loc}}^{2,\infty}(\hh^n)$. Let $\Om\subseteq\hh^n$ be a bounded open set. Let $R=R(\Om)>0$ be such that $\Om\Subset B(0,R)$. Then there exists $c_1=c_1(\Om)>0$ such that 
\begin{equation*}
\left|\frac{f(p \cdot \delta_\varepsilon(w)) - f(p) - \varepsilon \left\langle \nabla^\hhh f(p), w_\hhh \right\rangle}{\varepsilon^2}\right|\leq \mathrm{Lip}\left(\nabla^\hhh f,B(0,R+c_1)\right)d(0,w)^2
\end{equation*}
for every $p\in\Om$, $w\in B(0,1)$, $0<\varepsilon<c_1$.
\end{theorem}
\begin{proof}
Fix $w\in B(0,1)$.
Since $\Om\Subset B(0,R)$, there exists $c_1=c_1(\Om)>0$ such that $p\cdot\delta_\eps(w)\in B(0,R)$ for every $p\in\Om$ and every $0<\eps< c_1$. 
Let $\gamma:[0,1] \to \hh^n$ be an absolutely continuous, horizontal curve such that $\gamma(0)=p,\gamma(1)=p\cdot\delta_\varepsilon(w)$ and 
\begin{equation}\label{minicurve}
     d(p,p\cdot \delta_\varepsilon(w))=\int_0^1\sqrt{\left\langle\dot\gamma(\tau),\dot\gamma(\tau)\right\rangle}\,d\tau.
\end{equation}
The existence of such a curve is ensured e.g.~by \cite[Theorem 2.13]{MR3587666}.
Since $f\in W^{2,\infty}_{\hhh,\mathrm{loc}}(\hh^n)$, then $f\in C(\hh^n) $ and $Z_j f\in C(\hh^n)$ for any $j=1,\ldots,2n$ (cf.~\cite{MR1631642}). In particular, by \cite[Proposition 2.6]{MR4581339},
\begin{equation*}
f(p \cdot \delta_\varepsilon(w))-f(p)=\int_0^1 \left\langle \nabla^\hhh f(\gamma(\tau)),\dot{\gamma}(\tau) \right\rangle\,d\tau.
\end{equation*}
Moreover,
\begin{equation*}
\int_0^1 \left\langle \nabla^\hhh f(p),\dot{\gamma}(\tau) \right\rangle\,d\tau = \sum_{j=1}^{2n} Z_jf(p) \int_0^1 \dot{\gamma}_j(\tau)\,d\tau = \sum_{j=1}^{2n} Z_jf(p)(p_j+\varepsilon w_j - p_j)= \varepsilon \left\langle \nabla^\hhh f(p),w_H \right\rangle,
\end{equation*}
whence
\begin{equation*}
      f(p \cdot \delta_\varepsilon(w))-f(p)-\varepsilon \left\langle \nabla^\hhh f(p),w_H \right\rangle  = \int_0^1 \left\langle \nabla^\hhh f(\gamma(\tau))-\nabla^\hhh f(p),\dot{\gamma}(\tau) \right\rangle\,d\tau.
\end{equation*}
Notice that $\gamma(\tau)\in B(0,R+c_1)$ for any $\tau\in(0,1)$. Indeed, being $\gamma$ length-minimizing by \eqref{minicurve},
\begin{equation*}
    d(0,\gamma(\tau))\leq d(0,p)+d(p,\gamma(\tau))\leq d(0,p)+d(p,p\cdot\delta_\eps(w))< R+\eps d(0,w) < R+c_1.
\end{equation*}
Hence, since $\nabla^\hhh f$ is Lipschitz continuous on $B(0,R+c_1)$ and $\gamma$ is length-minimizing,
\begin{equation*}
\begin{split}
   \int_0^1 \left\langle \nabla^\hhh f(\gamma(\tau))-\nabla^\hhh f(p),\dot{\gamma}(\tau) \right\rangle\,d\tau
    & \leq \mathrm{Lip}\left(\nabla^\hhh f,B(0,R+c_1)\right)\int_0^1 d(p,\gamma(\tau))\abs*{\dot{\gamma}(\tau)}\,d\tau\\
    & \leq \mathrm{Lip}\left(\nabla^\hhh f,B(0,R+c_1)\right)d(p,p \cdot \delta_\varepsilon(w))\int_0^1 \abs*{\dot{\gamma}}\,d\tau\\
     \overset{\eqref{eq_CC_distance}}&{=} \mathrm{Lip}\left(\nabla^\hhh f,B(0,R+c_1)\right)\varepsilon^2d(w,0)^2,
\end{split}
\end{equation*}
which is the thesis.
\end{proof}

\section{Distributional solutions are renormalizable}\label{sec:commu}
In this section we prove that contact velocity fields with horizontal Sobolev regularity enjoy the renormalization property.
\subsection{Transport equation and distributional solutions} We begin by recalling some basic definitions and properties of transport equations.
 If $I\subseteq\rr$ is an open interval, we consider Borel functions $u:I\times\hh^n\to\rr$ up to $\mathcal{L}^{1+N}$-negligible sets. Inside this class, given
 $\bar\tau\in (0,\infty]$ and $\alpha,\beta\in [1,\infty]$, we consider (see \cite{MR3726909}) the space $L^\alpha\left(0,\bar\tau;L^\beta(\hh^n)\right)$ of those
 functions $u$ such that 
\begin{equation*}
\norm*{u}_{L^\alpha(0,\bar{\tau};L^\beta(\hh^n))}=\begin{cases}
\displaystyle\left(\int_0^{\bar{\tau}} \norm*{u(\cdot,\tau)}_{L^\beta(\hh^n)}^\alpha \,d\tau\right)^{\frac{1}{\alpha}} & \text{if $\alpha \in [1,\infty)$}\\
\mathop{\mathrm{ess\,sup}}\limits_{\tau \in (0,\bar{\tau})} \, \norm*{u(\cdot,\tau)}_{L^\beta(\hh^n)} & \text{if $\alpha=\infty$}
\end{cases}
\end{equation*}
is finite.
Its local version $L^\alpha\left(0,\bar{\tau};L^\beta_{\mathrm{loc}}(\hh^n)\right)$ is defined similarly. Moreover, we recall that when $\alpha=\beta$, $L^\alpha\left(0,\bar{\tau};L^\alpha(\hh^n)\right)$ naturally identifies with $L^\alpha\left((0,\bar{\tau})\times\hh^n\right)$.
If $\bb\in\VfH{n}$ is a velocity field, $c$ is a reaction term and $u_0$ is an initial condition, the Cauchy problem for the \emph{transport equation} associated with $\bb,c$ and $u_0$ reads formally as
\begin{equation} \label{eq_transport_equation}\tag{$\mathscr T$}
\begin{cases}
\displaystyle{\frac{\partial u}{\partial \tau} - \left\langle \bb,\nabla u \right\rangle + cu = 0} & \text{in $(0,\bar{\tau}) \times \hh^n$}\\
u(0,\cdot)=u_0 & \text{in $\hh^n$.}
\end{cases}
\end{equation}
The natural notion of solution to \eqref{eq_transport_equation} is that of \emph{distributional solution}. 
\begin{definition}[Distributional solutions]\label{def_distributional_solution}
    Fix $s \in [1,\infty]$ and assume that
\begin{equation} \label{b,c_assumptions}
\bb\in L^1\left(0,\bar{\tau};L^{s'}_{\mathrm{loc}}(\hh^n;\hh^n)\right), \qquad c,\divv \bb \in L^1\left(0,\bar{\tau};L^{s'}_{\mathrm{loc}}(\hh^n)\right),\qquad u_0\in L^s_{\mathrm{loc}}(\hh^n).
\end{equation}
We say that $u \in L^\infty\left(0,\bar{\tau};L^s_{\mathrm{loc}}(\hh^n)\right)$ ($u\in L^\infty\left((0,\bar\tau)\times\hh^n\right)$ if $s=\infty$) is a \emph{distributional solution} to \eqref{eq_transport_equation} if
\begin{equation*}
 \left\langle \mathscr T_{u_0,\bb,c}(u),\varphi\right\rangle\coloneqq-\int_{\hh^n}u_0(p)\varphi(0,p)\,dp+\int_{0}^{\bar{\tau}}\int_{\hh^n}u\left[-\partial_\tau\varphi+\langle\bb, \nabla \varphi \rangle+(c+\divv\bb)\varphi\right]\,dp\,d\tau=0 
\end{equation*}
for every $\varphi\in C_c^\infty([0,\bar{\tau})\times\hh^n)$.
\end{definition}
 Notice that, thanks to \eqref{b,c_assumptions} and to the $L^s-L^{s'}$ duality, \Cref{def_distributional_solution} is well-posed. Existence of solutions to \eqref{eq_transport_equation} is guaranteed, in great generality, by \cite[Proposition II.1]{MR1022305} (even though the result is stated in Euclidean spaces, we may apply it thanks to \eqref{eq:consistency_scalar} and the invariance of the divergence in the Euclidean and Heisenberg structures).
\begin{proposition} \label{prop_existence}
Let $s \in [1,\infty]$ and assume \eqref{b,c_assumptions}. Assume, in addition, that $u_0 \in L^s(\hh^n)$ and
\begin{equation*}
\begin{cases}
c+\frac{1}{s}\divv b \in L^1(0,\bar{\tau};L^\infty(\hh^n)) & \text{if $s \in (1,\infty]$}\\
c, \divv b \in L^1(0,\bar{\tau};L^\infty(\hh^n)) & \text{if $s=1$}.
\end{cases}
\end{equation*}
Then there exists a distributional solution $u \in L^\infty\left(0,\bar{\tau};L^s(\hh^n)\right)$ to \eqref{eq_transport_equation}.
\end{proposition}
The following simple remark will be crucial in the following (cf.~for instance~\cite{AST_2008__317__175_0}).
\begin{proposition}\label{propositiondelellis}
  Let $u \in L^\infty\left(0,\bar{\tau};L^s_{\mathrm{loc}}(\hh^n)\right)$ be any distributional solution to \eqref{eq_transport_equation}. Extend $\bb,c, u$ to $(-\infty,\bar\tau)\times\hh^n$ by 
   \begin{equation}\label{extensiondelellis}
\bb(\tau,p)=\begin{cases}
0 & \text{if $\tau<0$}\\
\bb(\tau,p) & \text{otherwise,}
\end{cases}
\qquad   
c(\tau,p)=\begin{cases}
0 & \text{if $\tau<0$}\\
c(\tau,p) & \text{otherwise,}
\end{cases}
\qquad 
u(\tau,p)=\begin{cases}
u_0(p) & \text{if $\tau<0$}\\
u(\tau,p) & \text{otherwise.}
\end{cases}
\end{equation}
Then 
\begin{equation*}
    \partial_\tau u - \left\langle \bb,\nabla u\right\rangle + cu = 0  \qquad\text{in $(-\infty,\bar{\tau}) \times \hh^n$}
\end{equation*}
in the sense of distributions, namely $\langle \mathscr T_{\bb,c}(u),\varphi\rangle=0$ for every $\varphi\in C_c^\infty((-\infty,\bar{\tau})\times\hh^n)$, with
\begin{equation}\label{equazioneestesasenzadato}
   \left\langle \mathscr T_{\bb,c}(u),\varphi\right\rangle\coloneqq \int_{-\infty}^{\bar{\tau}}\int_{\hh^n}u\left[-\partial_\tau\varphi+\langle \bb, \nabla \varphi \rangle+(c+\divv \bb)\varphi\right]\,dp\,d\tau.
\end{equation}
\begin{proof}
    Let $\varphi\in C_c^\infty((-\infty,\bar{\tau})\times\hh^n)$. Then
    \begin{equation*}
        \begin{split}
             \int_{-\infty}^{\bar{\tau}}\int_{\hh^n}u[-\partial_\tau\varphi+&\langle \bb, \nabla \varphi \rangle+(c+\divv \bb)\varphi]\,dp\,d\tau\\
             \overset{\eqref{extensiondelellis}}&{=} -\int_{\hh^n}u_0(p)\left(\int_{-\infty}^{0}\partial_\tau\varphi\,d\tau\right)\,dp+ \int_{0}^{\bar{\tau}}\int_{\hh^n}u\left[-\partial_\tau\varphi+\langle \bb, \nabla \varphi \rangle+(c+\divv \bb)\varphi\right]\,dp\,d\tau\\
              &=-\int_{\hh^n}u_0(p)\varphi(0,p)dp+ \int_{0}^{\bar{\tau}}\int_{\hh^n}u\left[-\partial_\tau\varphi+\langle \bb, \nabla \varphi \rangle+(c+\divv \bb)\varphi\right]\,dp\,d\tau\\
              \overset{\eqref{eq_transport_equation}}&{=}0.
        \end{split}
    \end{equation*}
\end{proof}
\end{proposition}
\subsection{Renormalized solutions}
Distributional solutions are in general not stable with respect to composition with test functions, not even in the divergence-free case,
see for instance \cite{DePauw_03} for a beautiful example.  DiPerna-Lions' notion of \emph{renormalized solution} to \eqref{eq_transport_equation} formalizes such a property.
\begin{definition}[Renormalized solutions]\label{def_renormalized_solution}
Let $s\in [1,\infty]$ and assume that \eqref{b,c_assumptions} holds with $u_0 \in L^s(\hh^n)$.
Then a function $u\in L^\infty\left(0,\bar\tau;L^s(\hh^n)\right)$ is a \emph{renormalized solution} to \eqref{eq_transport_equation} if, for every $\beta \in C^1(\rr)$ with
$\beta'$ bounded, the function $\beta(u)$ is a distributional solution to
\begin{equation} \label{eq_renormalized_solution}
\begin{cases}
\dfrac{\partial \beta(u)}{\partial \tau}-\left\langle \bb,\nabla \beta (u)\right\rangle + cu\beta' (u) = 0 & \text{in $(0,\bar{\tau}) \times\hh^n$}\\

\beta(u)(0,\cdot)=\beta(u_0) & \text{in $\hh^n$,}
\end{cases}
\end{equation}
namely if, for every $\varphi\in C^\infty_c\left([0,\bar\tau)\times\hh^n\right),$
\begin{equation*}
    -\int_{\hh^n}\beta(u_0)(p)\varphi(0,p)\,dp
    +\int_{0}^{\bar{\tau}}\int_{\hh^n}\beta(u)\left[-\partial_\tau \varphi+\left\langle \bb, \nabla \varphi\right \rangle+\divv \bb\,\varphi\right]+c u\beta'(u)\varphi\,dp\,d\tau.
\end{equation*}
If $s=\infty$ the condition that $\beta'$ is bounded is not required.
\end{definition}
Choosing $\beta$ to be the identity map in \Cref{def_renormalized_solution}, it is clear that renormalized solutions are distributional solutions. When the converse implication holds, $\bb$ is said to enjoy the \emph{renormalization property}.
\begin{definition}[Renormalization property]
 Assume that $\bb\in\VfH{n}$ satisfies \eqref{b,c_assumptions} with $s=\infty$. We say that $\bb$ has the \emph{renormalization property} if bounded distributional solutions to \eqref{eq_transport_equation} with
 $u_0 \in L^\infty(\hh^n)$ and $c\in L^1_{\mathrm{loc}}([0,\bar\tau)\times \hh^n)$ are renormalized solutions.
\end{definition}
The renormalization property of contact vector fields with horizontal Sobolev regularity will follow as a corollary of the main result of the paper. 
\begin{theorem} \label{teo_distributional_implies_renormalized}
Let $s \in [1,\infty]$ and $u_0 \in L^s(\hh^n)$. Assume that $\bb\in L^1\left(0,\bar\tau;\VfH{n}\right)$ is a time-dependent contact vector field, induced by $\psi\in L^1\left(0,\bar\tau;W^{2,s'}_\hhh(\hh^n)\right)$ according to \eqref{contact_vector_field}. Then, if $c\in L^1_{\mathrm{loc}}([0,\bar\tau)\times \hh^n)$, any distributional solution $u \in L^\infty(0,\bar{\tau};L^s(\hh^n))$ to \eqref{eq_transport_equation} 
is a renormalized solution. In particular, $\bb$ has the renormalization property.
\end{theorem}

As a consequence of \Cref{teo_distributional_implies_renormalized} and \Cref{prop_existence}, we get existence and uniqueness of distributional solutions to \eqref{eq_transport_equation}. 

\begin{theorem} \label{teo_uniqueness_null_solution}
Let $s \in [1,\infty]$, $u_0 \in L^s(\hh^n)$ and let $\bb$ as in \Cref{teo_distributional_implies_renormalized}. Assume, in addition, that $c, \divv \bb \in L^1\left(0,\bar{\tau};L^\infty(\hh^n)\right)$ and that
\begin{equation} \label{condizionedicrescitadpladattata}
\frac{|\bb|}{1+d(p,0)} \in L^1\left(0,\bar{\tau};L^1(\hh^n)\right)+L^1\left(0,\bar{\tau};L^\infty(\hh^n)\right),
\end{equation} 
where $|\bb|$ is the norm of $\bb$ with respect to the Riemannian metric $\langle\cdot,\cdot\rangle$.
Then there exists a unique distributional solution $u$ to \eqref{eq_transport_equation} in $L^\infty(0,\bar{\tau}; L^s(\hh^n))$ corresponding to the initial condition $u_0$.
\end{theorem}

\begin{remark}
    The integrability condition \eqref{condizionedicrescitadpladattata} is the natural adaptation to our framework, due to both the Riemannian norm of $\bb$ and the sub-Riemannian distance $d(p,0)$, of the \emph{Euclidean} growth assumption imposed in \cite{MR1022305}. Nevertheless, we stress that the latter would have worked as well in our setting: in that case, the proof of \Cref{teo_uniqueness_null_solution} follows \emph{verbatim} that of \cite[Theorem II.2]{MR1022305}.
\end{remark}

\subsection{Regularization}
As customary, the key step in the proof of the renormalization property consists in showing that the spatial group mollification $u*\rho_\varepsilon$ of a distributional solution 
$u$ to \eqref{eq_transport_equation} still satisfies \eqref{eq_transport_equation} up to a remainder distribution, the so-called \emph{commutator} $\mathscr C_\varepsilon$. As the latter admits an integral representation, the main effort is to prove that this error tends to $0$ not only in the sense of distributions, but actually in $L^1_{\mathrm{loc}}$. 
 We begin by providing an integral representation for the above-mentioned error term. 
\begin{proposition}\label{prop_formadelresto}
 For $s \in [1,\infty]$, assume \eqref{b,c_assumptions}. Fix a distributional solution $u \in L^\infty\left(0,\bar{\tau}; L^s(\hh^n)\right)$ to \eqref{eq_transport_equation} and extend $u,\bb,c$ to $(-\infty,\bar\tau)$ as in \eqref{extensiondelellis}. For every $\varepsilon>0$, let $u_\varepsilon = u*\rho_\varepsilon$ be the spatial group mollification of $u$ by $\rho_\varepsilon$, with $\rho_\varepsilon$ as in \eqref{eq_def_mollifier}. Then, the distribution $\mathscr T_{\bb,c}(u_\varepsilon)$ in \eqref{equazioneestesasenzadato} is representable by integration of the commutator $\mathscr C_\varepsilon$, namely
  \begin{equation*}
        \left\langle \mathscr T_{\bb,c}(u_\varepsilon),\varphi\right\rangle=\int_{-\infty}^{\bar\tau}\int_{\hh^n}\varphi(\tau,p)\mathscr C_\varepsilon(\tau,p)\,dp\,d\tau\qquad\text{for every $\varphi\in C^\infty_c((-\infty,\bar\tau)\times\hh^n)$,}
  \end{equation*}
  where
$\mathscr C_\varepsilon=\mathscr C_\varepsilon^1+\mathscr C_\varepsilon^2\in L^1\left(-\infty,\bar\tau,L_{\mathrm{loc}}^1(\hh^n)\right)$ is defined, for a.e.~$(\tau,p)\in(-\infty,\bar\tau)\times\hh^n$, by 
\begin{align*}
    \mathscr C_\eps^1(\tau,p)&\coloneqq -\left((u\divv \bb)*\rho_\varepsilon\right)(\tau, p )-\int_{\hh^n}u\left(\tau, p\cdot q^{-1} \right)\sum_{j=1}^N\left[b_j(\tau, p )Z_j\rho_\varepsilon(q)-b_j\left(\tau, p\cdot q^{-1} \right)Z^r_j\rho_\varepsilon(q)\right]\,dq,\\
     \mathscr C_\eps^2(\tau,p)&\coloneqq \int_{\hh^n}u\left(\tau, p\cdot q^{-1}\right )\rho_\varepsilon(q)\left[c(\tau, p )-c\left(\tau, p\cdot q^{-1}\right)\right]\,dq.
\end{align*}
\end{proposition}
\begin{proof}
 Fix $\varphi\in C^\infty_c\left((-\infty,\bar\tau)\times\hh^n\right)$. Since $u$ solves \eqref{eq_transport_equation}, then $u$ extended as in \eqref{extensiondelellis} solves \eqref{equazioneestesasenzadato} by \Cref{propositiondelellis}. Therefore,
\begin{equation}\label{mollizero}
    \left\langle \mathscr T_{\bb,c}(u)*\rho_\varepsilon,\varphi\right\rangle\overset{\eqref{deficonvdistr}}{=} \left\langle \mathscr T_{\bb,c}(u),\varphi*\check\rho_\varepsilon\right\rangle=0.
\end{equation}
Moreover, since $u_\varepsilon$ is smooth in the spatial variable and $Z_j$ is left-invariant, 
\begin{equation}\label{uepsmolliinproof}
\begin{split}
      \left\langle \mathscr T_{\bb,c}(u_\varepsilon),\varphi\right\rangle&= -\int_{-\infty}^{\bar{\tau}}\int_{\hh^n}u_\varepsilon\partial_\tau \varphi\,d p \,d\tau+\int_{-\infty}^{\bar{\tau}}\int_{\hh^n}\varphi\left[-\langle\bb,\nabla u_\varepsilon\rangle+cu_\varepsilon\right]\,d p \,d\tau\\
      &=-\int_{-\infty}^{\bar{\tau}}\int_{\hh^n}u_\varepsilon\partial_\tau \varphi\,d p \,d\tau+\int_{-\infty}^{\bar{\tau}}\int_{\hh^n}\varphi\left[-\sum_{j=1}^Nb_j\left(u*Z_j\rho_\varepsilon\right)+cu_\varepsilon\right]\,d p \,d\tau.
\end{split}
\end{equation}
On the other hand, one has
\begin{equation*}
\begin{split}
0\overset{\eqref{deficonvdistr},\eqref{mollizero}}&{=}
     \left\langle \mathscr T_{\bb,c}(u)*\rho_\varepsilon,\varphi\right\rangle\\
     &= -\int_{-\infty}^{\bar{\tau}}\int_{\hh^n}u(\tau, w )\frac{d}{d\tau}\int_{\hh^n}\varphi(\tau, p )\check\rho_\varepsilon\left( p^{-1}\cdot w \right)\,d p \,d w \,d\tau\\
     &\quad+\int_{-\infty}^{\bar{\tau}}\int_{\hh^n}u(\tau, w )\sum_{j=1}^N b_j(\tau, w )Z_j\left(\varphi*\check\rho_\varepsilon\right)(\tau, w )\,d w \,d\tau\\
     &\quad+\int_{-\infty}^{\bar{\tau}}\int_{\hh^n}\int_{\hh^n}u(\tau, w )(c+\divv\bb)(\tau, w )\varphi(\tau, p )\rho_\varepsilon\left( w ^{-1}\cdot p \right)\,d p \,d w \,d\tau\\
     \overset{\eqref{eq_derivatives_convolution},\eqref{scambiarederivateeinversione}}&{=}-\int_{-\infty}^{\bar{\tau}}\int_{\hh^n}u_\varepsilon(\tau, p )\partial_\tau \varphi(\tau, p )d p \,d\tau-\int_{-\infty}^{\bar{\tau}}\int_{\hh^n}u(\tau, w )\sum_{j=1}^N b_j(\tau, w )\left(\varphi*\widecheck {Z_j^r\rho_\varepsilon}\right)(\tau, w )\,d w \,d\tau\\
     &\quad+\int_{-\infty}^{\bar{\tau}}\int_{\hh^n}\int_{\hh^n}u(\tau, w )(c+\divv\bb)(\tau, w )\varphi(\tau, p )\rho_\varepsilon\left( w ^{-1}\cdot p \right)\,d p \,d w \,d\tau.   
\end{split}
\end{equation*}
In particular, we can replace in \eqref{uepsmolliinproof} the term $-\int_{-\infty}^{\bar{\tau}}\int_{\hh^n}u_\varepsilon\partial_\tau \varphi\,d p \,d\tau$ by the sum of two other terms resulting from the previous identity to get
\begin{equation*}
    \begin{split}
        &\left\langle \mathscr T_{\bb,c}(u_\varepsilon),\varphi\right\rangle  =-\int_{-\infty}^{\bar{\tau}}\int_{\hh^n}\varphi(\tau, p )\int_{\hh^n}u(\tau, w )\sum_{j=1}^Nb_j(\tau, p )Z_j\rho_\varepsilon\left( w ^{-1}\cdot p \right)\,d w \,d p \,d\tau\\
     &\quad+\int_{-\infty}^{\bar{\tau}}\int_{\hh^n}\varphi(\tau, p )\int_{\hh^n}u(\tau, w )\sum_{j=1}^N b_j(\tau, w )Z_j^r\rho_\varepsilon\left( w ^{-1}\cdot p \right)\,d w \,d p \,d\tau\\
     &\quad-\int_{-\infty}^{\bar{\tau}}\int_{\hh^n}\varphi(\tau, p )\left(\left(u\divv\bb\right)*\rho_\varepsilon\right)(\tau, p )\,d p \,d\tau+\int_{-\infty}^{\bar{\tau}}\int_{\hh^n}\varphi (\tau, p )\int_{\hh^n}u(\tau, w )c(\tau, p )\rho_\varepsilon\left( w^{-1}\cdot p \right)\,d w \,d p \,d\tau\\
     &\quad-\int_{-\infty}^{\bar{\tau}}\int_{\hh^n}\varphi(\tau, p )\int_{\hh^n}u(\tau, w )c(\tau, w )\rho_\varepsilon\left( w ^{-1}\cdot p \right)\,d w \,d p\,d\tau.
    \end{split}
\end{equation*}
Recalling \eqref{eq_Lebesgue_Haar}, the thesis follows by the change of variables $w\mapsto p\cdot q^{-1}$.
\end{proof}

Next, we show that, when $\bb$ has a contact structure, $\mathscr C_\eps\to 0$ strongly in $ L^1\left(-\infty,\bar\tau,L_{\mathrm{loc}}^1(\hh^n)\right)$. We stress that $\bb$ is a time-dependent contact vector field if and only if its extension as in \eqref{extensiondelellis} is a contact vector field (just choose $\psi(\tau,\cdot)\equiv 0$ for $\tau<0$).
\begin{lemma} \label{lem_commutation}
Let $s \in [1,\infty]$, let $\bb$ be a time-dependent contact vector field as in \Cref{teo_distributional_implies_renormalized} and let $c \in L^1\left(0,\bar{\tau};L^{s'}_{\mathrm{loc}}(\hh^n)\right)$. For $u \in L^\infty\left(0,\bar{\tau}; L^s_{\mathrm{loc}}(\hh^n)\right)$, extend $u,\bb,c$ to $(-\infty,\bar\tau)\times\hh^n$ as in \eqref{extensiondelellis}. Then  $\mathscr C_\varepsilon \to 0$ in $L^1\left(-\infty,\bar\tau; L^1_{\mathrm{loc}}(\hh^n)\right)$ as $\varepsilon \searrow 0$. 
\end{lemma}
\begin{proof}
Since every argument will be carried out for $\tau$ fixed, without loss of generality we assume that $u$, $b$ and $c$ do not depend on $\tau$. We prove
in two steps the local $L^1$-convergence of $\mathscr C^i_\varepsilon$.

\smallskip\noindent
{\bf Step 1. Convergence of $\mathscr C^1_\varepsilon$.} For $p\in\hh^n$ one has
\begin{equation*}
\begin{split}
    -\mathscr C_\varepsilon^1(p) &=\left({ (u\divv \bb)}*\rho_\varepsilon\right)( p )+\int_{\hh^n}u\left( p\cdot q^{-1} \right)\sum_{j=1}^{N}\left[b_j( p )Z_j\rho_\varepsilon(q)-b_j\left(p\cdot q^{-1} \right)Z^r_j\rho_\varepsilon(q)\right]\,d q\\
     \overset{\eqref{eq_right_vector_fields}}&{=} ((u \divv \bb)*\rho_\varepsilon)(p) \\
    &\quad+ \sum_{j=1}^n \int_{\hh^n} b_j(p) u\left(p \cdot q^{-1}\right) X_j^r \rho_\varepsilon(q) \,dq + 4\sum_{j=1}^n \int_{\hh^n} b_j(p) u\left(p \cdot q^{-1}\right) y_j(q) T \rho_\varepsilon(q) \,dq\\ & \quad + \sum_{j=1}^n \int_{\hh^n} b_{n+j}(p) u\left(p \cdot q^{-1}\right) Y_j^r \rho_\varepsilon(q) \,dq - 4\sum_{j=1}^n \int_{\hh^n} b_{n+j}(p) u\left(p \cdot q^{-1}\right) x_j(q) T \rho_\varepsilon(q) \,dq\\ & \quad + \int_{\hh^n} b_N(p) u\left(p \cdot q^{-1}\right) T \rho_\varepsilon(q) \,dq - \sum_{j=1}^N \int_{\hh^n} b_j\left(p \cdot q^{-1}\right) u\left(p \cdot q^{-1}\right) Z_j^r \rho_\varepsilon(q) \,dq.
\end{split}
\end{equation*}
In particular, recalling \eqref{eq_J},
\begin{equation*}
    \begin{split}
         -\mathscr C_\varepsilon^1(p) & = ((u \divv \bb)*\rho_\varepsilon)(p) +\sum_{j=1}^N \int_{\hh^n} \left[b_j(p) - b_j\left(p \cdot q^{-1}\right)\right] u\left(p \cdot q^{-1}\right) Z_j^r \rho_\varepsilon(q) \,dq\\ & \quad -4 \int_{\hh^n} \langle \bb(p), \JJ(q) \rangle u\left(p \cdot q^{-1}\right) T \rho_\varepsilon(q) \,dq\\
     \overset{\eqref{eq_derivative_mollification}}&{=} ((u \divv \bb)*\rho_\varepsilon)(p) + \sum_{j=1}^{2n} \int_{\hh^n} \left[b_j(p) - b_j\left(p \cdot q^{-1}\right)\right] u\left(p \cdot q^{-1}\right) \varepsilon^{-Q-1} Z_j^r \rho \left(\delta_{\frac{1}{\varepsilon}}(q)\right) \,dq\\ & \quad + \int_{\hh^n} \left[b_N(p) - b_N\left(p \cdot q^{-1}\right)\right] u\left(p \cdot q^{-1}\right) \varepsilon^{-Q-2} T \rho \left(\delta_{\frac{1}{\varepsilon}}(q)\right) \,dq\\ & \quad -4 \int_{\hh^n} \langle \bb(p), \JJ(q) \rangle u\left(p \cdot q^{-1}\right) \varepsilon^{-Q-2} T \rho \left(\delta_{\frac{1}{\varepsilon}}(q)\right) \,dq.
    \end{split}
\end{equation*}
Performing the change of variables $w = \delta_{\frac{1}{\varepsilon}}(q)$, and recalling that $Z_j^r\rho\left(w^{-1}\right)=-Z_j\rho(w)$ 
(here we use for the first time that $\rho$ is an even kernel) for every $j=1,\ldots,N$ by \Cref{lemscambiarederivateeinversione} and \eqref{eq_mollificatori}, we obtain
\begin{equation*}
\begin{split}
   - \mathscr C&_\varepsilon^1(p)  \overset{\eqref{eq_dilation_Lebesgue}}{=} ((u \divv \bb)*\rho_\varepsilon)(p) + \sum_{j=1}^{2n} \int_{\hh^n} \frac{b_j(p)-b_j\left(p \cdot \delta_\varepsilon\left(w^{-1}\right)\right)}{\varepsilon} u\left(p \cdot \delta_\varepsilon\left(w^{-1}\right)\right) Z_j^r \rho(w) \,dw\\ & \quad + \int_{\hh^n} \frac{b_N(p)-b_N\left(p \cdot \delta_\varepsilon\left(w^{-1}\right)\right)}{\varepsilon^2} u\left(p \cdot \delta_\varepsilon\left(w^{-1}\right)\right) T \rho(w) \,dw\\ & \quad -4 \int_{\hh^n} \frac{\langle\bb(p), \JJ(\delta_\varepsilon(w)) \rangle}{\varepsilon^2} u\left(p \cdot \delta_\varepsilon\left(w^{-1}\right)\right) T \rho(w) \,dw\\
    & = ((u \divv \bb)*\rho_\varepsilon)(p) + \underbrace{\sum_{j=1}^{2n} \int_{\hh^n} \frac{b_j(p \cdot \delta_\varepsilon(w))-b_j(p)}{\varepsilon} u(p \cdot \delta_\varepsilon(w)) Z_j \rho(w) \,dw}_{\eqqcolon A_1(p)}\\ & \quad + \underbrace{\int_{\hh^n} \frac{b_N(p \cdot \delta_\varepsilon(w))-b_N(p)}{\varepsilon^2} u(p \cdot \delta_\varepsilon(w)) T \rho(w) \,dw +4 \int_{\hh^n} \frac{\langle \JJ(\bb(p)), \delta_\varepsilon(w) \rangle}{\varepsilon^2} u(p \cdot \delta_\varepsilon(w)) T \rho(w) \,dw}_{\eqqcolon A_2(p)}.
\end{split}
\end{equation*}
First, we compute the $L^1_{\mathrm{loc}}$-limit of $A_1$. Recall that $b_j \in W_{\hhh,\mathrm{loc}}^{1,s'}(\hh^n)$ for every $j=1,\ldots,2n$. Therefore, if $s'<\infty$, the dominated convergence theorem and \Cref{thm_convergence_diff_quot} imply that, for every $K \Subset \hh^n$,
\begin{equation*}
\begin{split}
   \int_K & \left\vert \int_{\hh^n} \left[\frac{b_j(p \cdot \delta_\varepsilon(w))-b_j(p)}{\varepsilon} u(p \cdot \delta_\varepsilon(w)) - \left\langle\nabla^\hhh b_j(p), w_\hhh \right\rangle u(p) \right] Z_j \rho(w) \,dw \right\vert \,dp\\
    & \leq \int_{\hh^n} \abs*{Z_j \rho(w)} \left(\int_K \left\vert \frac{b_j(p \cdot \delta_\varepsilon(w))-b_j(p)}{\varepsilon} u(p \cdot \delta_\varepsilon(w)) - \left\langle \nabla^\hhh b_j(p), w_\hhh \right\rangle u(p) \right\vert \,dp\right) \,dw \xrightarrow[\varepsilon \searrow 0]{} 0,
\end{split}
\end{equation*}
where we also used the $L^s$-continuity of right translations to replace $u(p \cdot \delta_\varepsilon(w))$ with $u(p)$ before applying
\Cref{thm_convergence_diff_quot}.
If instead $s'=\infty$, fix $j=1,\ldots,2n$ and set, for every $\varepsilon>0$,
\begin{equation*}
f_\varepsilon(p,w)=\frac{b_j(p \cdot \delta_\varepsilon(w))-b_j(p)}{\varepsilon}, \qquad f(p,w)=\left\langle \nabla^\hhh b_j(p), w_\hhh \right\rangle.
\end{equation*}
Fix open sets $K \Subset\Om\Subset \hh^n$.
Then
\begin{equation*}
\begin{split}
    \int_K & \abs*{\int_{\hh^n} [f_\varepsilon(p,w)u(p \cdot \delta_\varepsilon(w))-f(p,w)u(p)]Z_j\rho(w)\,dw}\,dp\\
    & \leq \int_{B(0,1)} \abs*{Z_j\rho(w)} \left(\int_K \abs*{f_\varepsilon(p,w)u(p \cdot \delta_\varepsilon(w))-f(p,w)u(p)}\,dp\right)\,dw.
\end{split}
\end{equation*} 
 By \eqref{lipboundfirstorder}, 
\begin{equation*}
    |f_\eps(p,w)|\leq \mathrm{Lip}(b_j,\Om)d(0,w),\qquad|f(p,w)|\leq|w_\hhh| \left\|\nabla^\hhh b_j\right\|_{L^\infty(K)}
\end{equation*}
for every $p\in K$, $w\in B(0,1)$ and $0<\eps<c_1$. Therefore, for $w\in B(0,1)$ and $0<\eps<c_1$ fixed,
\begin{equation*}
\begin{split}
    \int_K &\abs*{f_\varepsilon(p,w)u(p \cdot \delta_\varepsilon(w))-f(p,w)u(p)}\,dp\\ 
    &    \leq \int_K \abs*{f_\varepsilon(p,w)}\abs*{u(p \cdot \delta_\varepsilon(w))-u(p)}\,dp + \int_K \abs*{f_\varepsilon(p,w)-f(p,w)}\abs*{u(p)}\,dp\\
    & \leq\mathrm{Lip}(f,\Om)d(0,w)\norm*{u \circ R_{\delta_\varepsilon(w)}-u}_{L^1(K)}+\int_K \abs*{f_\varepsilon(p,w)-f(p,w)}\abs*{u(p)}\,dp.
\end{split}
\end{equation*}
The first term converges to $0$ thanks to \Cref{lem_continuity_translations}. For the second one, 
\begin{equation*}
\abs*{f_\varepsilon(\cdot,w)-f(\cdot,w)}\abs*{u} \leq \left(\mathrm{Lip}(f,\Om)d(0,w)+|w_\hhh| \left\|\nabla^\hhh b_j\right\|_{L^\infty(K)}\right)\abs*{u} \in L^1(K).
\end{equation*}
Moreover, since $b_j\in W^{1,s'}_\hhh(\Om)$ for every $s'\in[1,\infty]$, \Cref{thm_convergence_diff_quot} implies that, up to a subsequence, $\abs*{f_\varepsilon(\cdot,w)-f(\cdot,w)}\abs*{u} \to 0$ as $\varepsilon \searrow 0$ a.e.~on $K$. Hence, by the dominated convergence theorem, 
\begin{equation*}
K_\varepsilon(w) \coloneqq \int_K \abs*{f_\varepsilon(p,w)u(p \cdot \delta_\varepsilon(w))-f(p,w)u(p)}\,dp \xrightarrow[\varepsilon \searrow 0]{} 0.
\end{equation*}
In addition, by the above computations, $K_\eps$ is bounded in $L^\infty(B(0,1))$ uniformly in $\eps\in (0,c_1)$. Therefore, we conclude by the dominated convergence theorem that
\begin{equation*}
\begin{split}
    \int_K & \left\vert \int_{\hh^n} \left[\frac{b_j(p \cdot \delta_\varepsilon(w))-b_j(p)}{\varepsilon} u(p \cdot \delta_\varepsilon(w)) - \left\langle\nabla^\hhh b_j(p), w_\hhh \right\rangle u(p) \right] Z_j \rho(w) \,dw \right\vert \,dp\\
    & = \int_K \abs*{\int_{\hh^n} [f_\varepsilon(p,w)u(p \cdot \delta_\varepsilon(w))-f(p,w)u(p)]Z_j\rho(w)\,dw}\,dp \xrightarrow[\varepsilon \searrow 0]{} 0.
\end{split}
\end{equation*}
In any case, we have obtained that
\begin{equation*}
\begin{split}
    &\lim_{\varepsilon \searrow 0}^{L^1_{\mathrm{loc}}} A_1(p)  = \sum_{j=1}^{2n} \int_{\hh^n} \left\langle \nabla^\hhh b_j(p), w_\hhh \right\rangle u(p) Z_j\rho(w) \,dw\\
    &\quad = \sum_{j=1}^{2n} \sum_{\substack{i=1 \\ i \neq j}}^{2n} \int_{\hh^n} Z_i b_j(p) w_i u(p) Z_j \rho(w) \,dw + \sum_{j=1}^{2n} \int_{\hh^n} Z_j b_j(p) w_j u(p) Z_j \rho(w) \,dw\\
    & \quad= \sum_{j=1}^{2n} \sum_{\substack{i=1 \\ i \neq j}}^{2n} Z_i b_j(p) u(p) \underbrace{\int_{\hh^n} Z_j(w_i \rho(w)) \,dw}_{=0} + u(p)\sum_{j=1}^{2n} Z_j b_j(p)    \Bigg(\underbrace{\int_{\hh^n} Z_j(w_j \rho(w)) \,dw}_{=0} -\int_{\hh^n} \rho(w) \underbrace{Z_j w_j}_{=1} \,dw\Bigg)\\
    & \quad= -u(p)\sum_{j=1}^{2n} Z_j b_j(p)=- u(p)\divv\bb(p)+u(p)T b_N(p).
\end{split}
\end{equation*}
Next, we focus on $A_2$. Notice that
\begin{equation*}
\begin{split}
    A_2(p) & = \underbrace{\int_{\hh^n} \frac{b_N(p \cdot \delta_\varepsilon(w))-b_N(p) - \varepsilon \langle \nabla^\hhh b_N(p), w_\hhh\rangle}{\varepsilon^2} u(p \cdot \delta_\varepsilon(w)) T \rho(w) \,dw}_{\eqqcolon B_1(p)} \\ & \quad + \underbrace{\int_{\hh^n} \frac{\langle \nabla^\hhh b_N(p) + 4\JJ(\bb(p)), \delta_\varepsilon(w) \rangle}{\varepsilon^2} u(p \cdot \delta_\varepsilon(w)) T \rho(w) \,dw}_{\eqqcolon B_2(p)}.
\end{split}
\end{equation*}
Recall that, since $\bb$ is a contact vector field, $b_N \in W^{2,s'}_{\hhh,\mathrm{loc}}(\hh^n)$. Therefore, arguing as above, we get, by \Cref{thm_convergence_2nd_diff_quot} if $s'<\infty$ or \Cref{thm_convergence_2nd_diff_quot_infty} if $s'=\infty$,
\begin{equation*}
\begin{split}
    \lim_{\varepsilon \searrow 0}^{L^1_{\mathrm{loc}}} B_1(p) & = \int_{\hh^n} \left(\frac{1}{2} \left\langle \nabla^{2,\hhh} b_N(p) w_\hhh, w_\hhh \right\rangle + w_NTb_N(p) \right) u(p) T\rho(w) \,dw\\
    & = \sum_{i,j=1}^{2n} \frac{1}{2} \int_{\hh^n} Z_i Z_j b_N(p) w_i w_j u(p) T\rho(w) \,dw + \int_{\hh^n} w_NTb_N(p) u(p) T\rho(w) \,dw\\
    & = \sum_{i,j=1}^{2n} \frac{1}{2}  Z_i Z_j b_N(p) u(p) \underbrace{\int_{\hh^n} T(w_i w_j \rho(w)) \,dw}_{=0}\\ & \quad + Tb_N(p) u(p) \Biggl(\underbrace{\int_{\hh^n} T(w_N \rho(w)) \,dw}_{=0} -\int_{\hh^n} \rho(w) \underbrace{T w_N}_{=1} \,dw\Biggr)\\
    & = -Tb_N(p) u(p).
\end{split}
\end{equation*}
Finally, since $\bb$ is a contact vector field generated by $\psi\in W^{2,s'}_\hhh(\hh^n)$,
\begin{equation}\label{extra_J}
\nabla^\hhh b_N + 4\JJ(\bb) \overset{\eqref{contact_vector_field}}{=} 
-4\nabla^\hhh {\psi} +4\JJ\left(-\JJ\left(\nabla^\hhh {\psi}\right)\right) \overset{\eqref{eq_J_properties}}{=} 0.
\end{equation}
Then $B_2=0$. In conclusion,
\begin{equation}\label{eq_primorestovaazero}
-\lim_{\varepsilon \searrow 0}^{L^1_{\mathrm{loc}}} \mathscr C_\varepsilon^1 = \lim_{\varepsilon \searrow 0}^{L^1_{\mathrm{loc}}} (u\divv \bb)*\rho_\varepsilon + \lim_{\varepsilon \searrow 0}^{L^1_{\mathrm{loc}}} (A_1+A_2)= u\divv \bb-u\divv \bb=0,
\end{equation}
where the second equality holds due to \Cref{prop_group_mollification}, since $u\divv \bb \in L^1_{\mathrm{loc}}(\hh^n)$. 

\smallskip\noindent
{\bf Step 2. Convergence of $\mathscr C^2_\varepsilon$.} For $p\in\hh^n$ one has
\begin{equation*}
\begin{split}
    \mathscr C_\varepsilon^2(p) 
    &=\int_{\hh^n} \left[c(p)u\left(p \cdot q^{-1}\right)\rho_\varepsilon(q)-c\left(p \cdot q^{-1}\right)u\left(p \cdot q^{-1}\right)\rho_\varepsilon(q)\right]\,dq\\
    &=\int_{\hh^n}\left[c(p)-c\left(p \cdot q^{-1}\right)\right]u\left(p \cdot q^{-1}\right)\varepsilon^{-Q}\rho\left(\delta_{\frac{1}{\varepsilon}}(q)\right)\,dq\\
    \overset{\eqref{eq_dilation_Lebesgue}}&{=}\int_{\hh^n}[c(p)-c(p \cdot \delta_\varepsilon(w))]u(p \cdot \delta_\varepsilon(w))\rho(w)\,dw.
\end{split}
\end{equation*}
Therefore, fixing $K \Subset\hh^n$, we obtain by H\"older's inequality
\begin{equation*}
\begin{split}
    \norm*{\mathscr C_\varepsilon^2}_{L^1(K)} & \leq \int_K \int_{\hh^n} \abs*{c(p \cdot \delta_\varepsilon(w))-c(p)}\abs*{u(p \cdot \delta_\varepsilon(w))}\rho(w) \,dw\,dp\\
    & = \int_{\hh^n} \rho(w) \left(\int_K \abs*{c(p \cdot \delta_\varepsilon(w))-c(p)}\abs*{u(p \cdot \delta_\varepsilon(w))}\,dp\right)\,dw\\
    & \leq \int_{\hh^n} \rho(w) \norm*{c \circ R_{\delta_\varepsilon(w)} - c}_{L^{s'}(K)} \norm*{u \circ R_{\delta_\varepsilon(w)}}_{L^s(K)} \,dw\\
     \overset{\eqref{eq_Lebesgue_Haar}}&{=} \norm*{u}_{L^s(K)} \int_{\hh^n} \rho(w) \norm*{c \circ R_{\delta_\varepsilon(w)} - c}_{L^{s'}(K)} \,dw.
\end{split}
\end{equation*}
Since $\norm*{c \circ R_{\delta_\varepsilon(w)} - c}_{L^{s'}(K)} \xrightarrow[\varepsilon \searrow 0]{} 0$ by \Cref{lem_continuity_translations} and
\begin{equation*}
\rho(w) \norm*{c \circ R_{\delta_\varepsilon(w)} - c}_{L^{s'}(K)} \overset{\eqref{eq_Lebesgue_Haar}}{\leq} 2\norm*{c}_{L^{s'}(K)} \rho(w) \in L^1(\hh^n) \qquad \text{for every $\varepsilon>0$ and $w \in \hh^n$,}
\end{equation*}
$\mathscr C_\varepsilon^2 \xrightarrow[\varepsilon \searrow 0]{L^1_{\mathrm{loc}}} 0$ by the dominated convergence theorem.
\end{proof}
\subsection{Proofs of \Cref{teo_distributional_implies_renormalized} and \Cref{teo_uniqueness_null_solution}}

\begin{proof}[Proof of \Cref{teo_distributional_implies_renormalized}] Extend $u,\bb,c$ to $(-\infty,\bar\tau)$ as in \eqref{extensiondelellis}.
Let $u_\eps$ and $\mathscr C_\eps$ be as in \Cref{prop_formadelresto}, so that $u_\eps$ is a distributional solution to 
\begin{equation*} 
\partial_\tau u_\eps - \left\langle \bb,\nabla u_\eps\right\rangle + cu_\eps =\mathscr C_\eps  \qquad \text{in $(-\infty,\bar{\tau}) \times\hh^n$,}
\end{equation*}
that is, recalling the smoothness of $u_\eps $ in the spatial variable,
\begin{equation}\label{dertempdebinproof}
\begin{split}
   -\int_{-\infty}^{\bar{\tau}}\int_{\hh^n}u_\eps\partial_\tau \varphi\,dp\,d\tau=\int_{-\infty}^{\bar{\tau}}\int_{\hh^n}\left(\left\langle \bb, \nabla u_\eps\right \rangle-u_\eps c+\mathscr C_\eps\right)\varphi \,dp\,d\tau\qquad\text{for every $\varphi\in C^\infty_c\left((-\infty,\bar\tau)\times\hh^n\right)$.}
\end{split}
\end{equation}
 Since $\left\langle\bb, \nabla u_\eps\right \rangle-u_\eps c+\mathscr C_\eps\in L^1_{\mathrm{loc}}\left((-\infty,\bar\tau)\times \hh^n\right)$, \eqref{dertempdebinproof} implies that $\partial_\tau u_\varepsilon \in L^1_{\mathrm{loc}}\left((-\infty,\bar\tau)\times \hh^n\right)$, and that 
 \begin{equation}\label{eq_regolarized_solution}
     \partial_\tau u_\eps - \left\langle\bb, \nabla u_\eps\right\rangle + cu_\eps = \mathscr C_\eps\qquad\text{a.e.~on $(-\infty,\tau)\times \hh^n$.}
 \end{equation}
Let $\beta \in C^1(\rr)$ be such that $|\beta'| \leq M<\infty$. In particular, $\beta'(u_\eps)\in L^\infty\left((-\infty,\bar\tau)\times \hh^n\right)$. Multiplying \eqref{eq_regolarized_solution} by $\beta'(u_\eps)$ and applying the chain rule $\eqref{chainrulesobolev}$, we deduce that 
\begin{equation}\label{aeperbetaeps}
    \frac{\partial \beta(u_\eps)}{\partial \tau} - \left\langle \bb,\nabla \beta(u_\eps)\right\rangle + cu_\eps\beta'(u_\eps) = \beta'(u_\eps)
    \mathscr C_\eps\qquad\text{a.e.~on $(-\infty,\tau)\times \hh^n$.}
\end{equation}
Fix $\varphi\in C^\infty_c\left((-\infty,\bar\tau)\times \hh^n\right)$. Multiplying \eqref{aeperbetaeps} by $\varphi$ and integrating by parts, we get
\begin{equation*}
      \int_{-\infty}^{\bar{\tau}}\int_{\hh^n}\left(-\beta(u_\eps)\partial_\tau \varphi+\beta(u_\eps)\left(\left\langle \bb, \nabla \varphi\right \rangle+\divv\bb\,\varphi\right)+c u_\eps\beta'(u_\eps)\varphi\right)\,dp\,d\tau=\int_{-\infty}^{\bar{\tau}}\int_{\hh^n}\beta'(u_\eps)\mathscr C_\epsilon\varphi\,dp\,d\tau.
\end{equation*}
First, by \Cref{lem_commutation} and recalling that $ \beta'$ is bounded,
\begin{equation*}
    \int_{-\infty}^{\bar{\tau}}\int_{\hh^n}\beta'(u_\eps)\mathscr C_\eps\varphi\,dp\,d\tau\xrightarrow[\eps\searrow 0]{}0.
\end{equation*}
For the left-hand side, observe first that, by the dominated convergence theorem,
\begin{equation*}
\begin{split}
    \left\vert \int_{-\infty}^{\bar{\tau}} \int_{\hh^n} \left(\beta (u_\varepsilon)-\beta( u)\right)\partial_\tau \varphi\,dp\,d\tau \right\vert
     \leq M\int_{-\infty}^{\bar{\tau}} \norm*{\partial_\tau \varphi}_{L^{s'}(\hh^n)} \norm*{u_\varepsilon - u}_{L^s(\hh^n)}\,d\tau \xrightarrow[\varepsilon \searrow 0]{} 0
\end{split}
\end{equation*}
since $u_\varepsilon (\tau,\cdot)\xrightarrow[\varepsilon \searrow 0]{L^s} u(\tau,\cdot)$ for a.e.~$\tau\in (-\infty,\bar\tau)$ by \Cref{prop_group_mollification} and
\begin{equation*}
\norm*{\partial_\tau \varphi}_{L^{s'}(\hh^n)} \norm*{u_\varepsilon - u}_{L^s(\hh^n)} \overset{\eqref{eq_Young_inequality}}{\leq} 2\norm*{\partial_\tau \varphi}_{L^{s'}(\hh^n)} \norm*{u}_{L^s(\hh^n)} \in L^1(0,\bar{\tau}) \qquad \text{for every $\varepsilon>0$.}
\end{equation*}
Let $\Om\Subset\hh^n$ be such that $\mathrm{supp\,}\varphi\subseteq(-\infty,\bar\tau)\times\Om$. Then, arguing as above,
\begin{equation*}
        \left| \int_{-\infty}^{\bar{\tau}}\int_{\hh^n}\left(\beta(u_\eps)-\beta(u)\right)\left\langle \bb, \nabla \varphi\right \rangle\,dp\,d\tau\right|\leq
        M\int_{-\infty}^{\bar{\tau}} \norm*{|\bb|}_{L^{s'}(\Om)} \norm*{|\nabla \varphi|}_{L^{\infty}(\Om)}\norm*{u_\varepsilon - u}_{L^s(\Om)}\,d\tau 
        \xrightarrow[\varepsilon \searrow 0]{} 0
\end{equation*}
and
\begin{equation*}
        \left| \int_{-\infty}^{\bar{\tau}}\int_{\hh^n}\left(\beta(u_\eps)-\beta(u)\right)\divv \bb\, \varphi\,dp\,d\tau\right|\leq M
         \int_{-\infty}^{\bar{\tau}} \norm*{\varphi}_{L^{\infty}(\hh^n)}\norm*{\divv\bb}_{L^{s'}(\Om)} \norm*{u_\varepsilon - u}_{L^s(\Om)}\,d\tau 
         \xrightarrow[\varepsilon \searrow 0]{} 0.
\end{equation*}
Finally, fix any subsequence of $\{u_\eps\}_{\eps>0}$, still denoted by $\{u_\eps\}_{\eps>0}$. Recall that $u_\varepsilon \xrightarrow[\varepsilon \searrow 0]{L^s} u$ by \Cref{prop_group_mollification}, so $u_\varepsilon \xrightarrow[\varepsilon \searrow 0]{} u$ pointwise a.e.~in $(0,\bar\tau)\times\hh^n$ up to a further subsequence. By the continuity of $\beta'$, $\beta'(u_\varepsilon)u_\eps- \beta'(u)u\xrightarrow[\varepsilon \searrow 0]{}0 $ pointwise a.e. Moreover, arguing as above, 
\begin{equation*}
    \|\beta' (u_\varepsilon)u_\varepsilon - \beta'(u)u]c\varphi\|_{L^1(\Om)}\leq  3M\norm*{\varphi}_{L^\infty(\hh^n)}\norm*{c}_{L^{s'}(\Om)} \norm*{u}_{L^{s}(\Om)},
\end{equation*}
whence
\begin{equation*}
    \left\vert \int_{-\infty}^{\bar{\tau}} \int_{\hh^n} \left[\beta' (u_\varepsilon)u_\varepsilon - \beta'(u)u\right]c\varphi\,dp\,d\tau \right\vert\xrightarrow[\varepsilon \searrow 0]{} 0.
    \end{equation*}
In conclusion, we proved that, for every $\varphi\in C^\infty_c\left((-\infty,\bar\tau)\times\hh^n\right)$,
\begin{equation}\label{prefineblim}
      \int_{-\infty}^{\bar{\tau}}\int_{\hh^n}\left[-\beta(u)\partial_\tau \varphi+\beta(u)\left(\left\langle b, \nabla \varphi\right \rangle+\divv b\varphi\right)+c u\beta'(u)\varphi\right]\,dp\,d\tau=0.
\end{equation}

Fix $\varphi\in C^\infty_c\left([0,\bar\tau)\times\hh^n\right)$, and extend it in such a way that $\varphi\in C^\infty_c\left((-\infty,\bar\tau)\times\hh^n\right)$. Then
\begin{equation*}
    \begin{split}
        0\overset{\eqref{prefineblim}}&{=}\int_{-\infty}^{\bar{\tau}}\int_{\hh^n}\left[-\beta(u)\partial_\tau \varphi+\beta(u)\left(\left\langle \bb, \nabla \varphi\right \rangle+\divv \bb\,\varphi\right)+c u\beta'(u)\varphi\right]\,dp\,d\tau\\
        &= -\int_{\hh^n}\beta(u_0)(p)\left(\int_{-\infty}^0\partial_\tau \varphi\,d\tau\right)\,dp+\int_{0}^{\bar{\tau}}\int_{\hh^n}\left[-\beta(u)\partial_\tau \varphi+\beta(u)\left(\left\langle \bb, \nabla \varphi\right \rangle+\divv \bb\,\varphi\right)+c u\beta'(u)\varphi\right]\,dp\,d\tau\\
        &\quad -\int_{\hh^n}\beta(u_0)(p)\varphi(0,p)\,dp+\int_{0}^{\bar{\tau}}\int_{\hh^n}\left[-\beta(u)\partial_\tau \varphi+\beta(u)\left(\left\langle \bb, \nabla \varphi\right \rangle+\divv \bb\,\varphi\right)+c u\beta'(u)\varphi\right]\,dp\,d\tau,
    \end{split}
\end{equation*}
where the last equality follows by \eqref{extensiondelellis}. Therefore, $u$ is a renormalized solution. When $s=\infty$, the previous proof applies directly for every $\beta \in C^1(\rr)$, because $\beta'$ is bounded on 
$\left[-\norm*{u}_{L^\infty(\hh^n)},\norm*{u}_{L^\infty(\hh^n)}\right]$ (recall \eqref{eq_Young_inequality}). The proof of \Cref{teo_distributional_implies_renormalized} is concluded.
\end{proof}

\begin{proof}[Proof of \Cref{teo_uniqueness_null_solution}] 
Owing to \Cref{teo_distributional_implies_renormalized}, it follows \emph{verbatim} that of  \cite[Theorem II.2]{MR1022305}, with the only difference that one has to use cut-off functions $\phi_R$ adapted
to the $\hh^n$ geometry, namely $\phi_R(p)=\phi(\delta_{1/R}(p))$
with $\phi\in C^\infty_c(\hh^n)$ identically equal to 1 on the
ball $B(0,1)$.
\end{proof}
\section{Consequences of the renormalization property}\label{sec:consequences}
In this section, we briefly recall the most relevant consequences of the renormalization property.
\subsection{Transport equation with measurable initial conditions}
If one wishes to deal with \eqref{eq_transport_equation} imposing only measurability on the initial condition, the notion of distributional solution to \eqref{eq_transport_equation} is no longer meaningful. To this aim, the authors of \cite{MR1022305} defined solutions by imposing the validity of the renormalization property with a careful choice of test functions.
We limit ourselves to provide the main definitions and statements, referring to \cite{MR1022305} for further details and more exhaustive results. 

The set of functions $L^0(\hh^n)$ is defined by
\begin{equation*}
L^0(\hh^n)=\left\{u:\hh^n \to [-\infty,\infty] \text{ measurable} : \abs*{u}<\infty \text{ a.e., }\mathcal{L}^N(\{\abs*{u}>\lambda\})<\infty \,\, \text{for every }\,\lambda>0\right\}.
\end{equation*}
For a precise definition of $L^\infty\left(0,\bar{\tau};L^0(\hh^n)\right)$, see \cite{MR1022305}. 
Observe that, if $u \in L^0(\hh^n)$ and $\beta \in C(\rr) \cap L^\infty(\rr)$ vanishes near $0$, 
then $\beta(u)\in L^1(\hh^n) \cap L^\infty(\hh^n)$. Hence, the following definition is well-posed.

\begin{definition}
Let $u_0 \in L^0(\hh^N)$ and assume that $\bb\in\VfH{n}$ and $c$ satisfy
\begin{equation} \label{eq_b,c_conditions}
\begin{split}
    &c, \divv \bb \in L^1\left(0,\bar{\tau};L^\infty(\hh^n)\right),\\  
    &\frac{\abs*{\bb}}{1+d(p,0)} \in L^1\left(0,\bar{\tau};L^1(\hh^n)\right)+L^1\left(0,\bar{\tau};L^\infty(\hh^n)\right).
\end{split}
\end{equation}
We say that $u \in L^\infty\left(0,\bar{\tau};L^0(\hh^n)\right)$ is a \emph{renormalized solution} to \eqref{eq_transport_equation} if $\beta (u)$ solves \eqref{eq_renormalized_solution} in the sense of distributions for every $\beta \in C^1(\rr) \cap L^\infty(\rr)$ such that $\beta$ vanishes near $0$ and $\beta'(1+\abs*{t}) \in L^\infty(\rr)$. 
\end{definition}
The following result (cf.~\cite[Theorem II.3]{MR1022305} and \cite[Theorem II.4]{MR1022305}) ensures the well-posedness of the transport equation with initial condition in $L^0$.  

\begin{theorem}
Assume \eqref{eq_b,c_conditions} and that $\bb\in\VfH{n}$ is a contact vector field as in \Cref{teo_distributional_implies_renormalized}. Then there exists a unique renormalized solution $u \in L^\infty\left(0,\bar{\tau};L^0(\hh^n)\right)$ to \eqref{eq_transport_equation} corresponding to any initial condition $u_0\in L^0(\hh^n)$. 
Furthermore, $u$ belongs to $C\left([0,\bar{\tau}];L^0(\hh^n)\right)$ and is stable under perturbations of the data.
\end{theorem}
\subsection{Regular Lagrangian Flows}
We sketch how this theory applies to provide a robust existence and uniqueness theory for ODE's (properly understood) associated
to weakly differentiable velocity fields, besides \cite{MR1022305,MR2096794}, see also \cite{MR3283066} for Lecture Notes on this topic. Here we follow the axiomatization introduced in \cite{MR2096794}, based on the
continuity equation
\begin{equation}\label{eq_continuity_equation}
\displaystyle{\frac{\partial u}{\partial \tau} +\divv(\bb u)= 0} \qquad \text{in $(0,\bar{\tau}) \times \hh^n$}
\end{equation}
and the induced semigroup $\tau\mapsto u(\tau,\cdot)$, thought when $u\geq 0$ as a semigroup in the space of probability densities,
rather than the original one of \cite{MR1022305}, based on the transport equation (see \cite[Remark~6.7]{MR2096794} for a more precise comparison
between the two approaches, the one based on \eqref{eq_continuity_equation} being more suitable to consider singular cases when
$\divv\bb$ is not absolutely continuous). In the classical setting, via the theory of characteristics, the ODE
\begin{equation} \label{eq_ode}
\begin{cases}
\displaystyle{\frac{d}{d\tau} \Phi(\tau,p)=\bb(\tau,\Phi(\tau,p))}\\
\Phi(0,p)=p.
\end{cases}
\end{equation}
is related to the transport equation \eqref{eq_transport_equation} (with $\bb$ replaced by $-\bb$) 
\begin{equation}\label{eq_new_transport_equation}
\displaystyle{\frac{\partial u}{\partial \tau} + \left\langle \bb,\nabla u \right\rangle + cu = 0} \qquad \text{in $(0,\bar{\tau}) \times \hh^n$}
\end{equation} 
and to the continuity equation \eqref{eq_continuity_equation} respectively by
\begin{equation}
u(s,\Phi(s,p))=u_0(p)-\int_0^s c(\tau,\Phi(\tau,p))\,d\tau\quad\forall s\in [0,\bar\tau),
\qquad
u(s,\Phi(s,p))=\frac{u(0,p)}{{\rm det}\nabla_p\Phi(s,p)} \quad\forall s\in [0,\bar\tau),
\end{equation} 
while the relation between the two PDE's simply comes with the choice $c=\divv\bb$.

\begin{definition} \label{def:RLF} We say that $\Phi(\tau,p)$ is a Regular Lagrangian Flow (RLF) associated to $\bb$ if
\begin{itemize}
\item[(a)] for a.e.~$p$ one has 
$$\Phi(s,p)=p+\int_0^s\bb(\tau,\Phi(\tau,p))\,d\tau\qquad\text{for all $s\in [0,\bar\tau)$;}$$
\item[(b)] there exists a constant $C\in (0,\infty)$ such that $\Phi(s,\cdot)_\#\mathcal{L}^N\leq C\mathcal{L}^N$ for all $s\in [0,\bar\tau)$.
\end{itemize}
\end{definition}

The main result in \cite{MR2096794}  (here we adapt it to the setting
of $\hh^n$) is that, under the growth condition
\begin{equation*}
\frac{|\bb|}{1+d(p,0)} \in L^1\left(0,\bar{\tau};L^1(\hh^n)\right)+L^1\left(0,\bar{\tau};L^\infty(\hh^n)\right)
\end{equation*}
on $\bb$, existence and uniqueness for \eqref{eq_continuity_equation}
for any nonnegative initial condition $u_0\in L^\infty(\hh^n)$
in the class of nonnegative $u$ in $L^\infty_{\mathrm{loc}}\left([0,\tau);L^\infty(\hh^n)\right)
\cap L^\infty_{\mathrm{loc}}\left([0,\tau);L^1(\hh^n)\right)$ 
is equivalent to existence and uniqueness of the RLF. In particular,  \Cref{teo_uniqueness_null_solution} applied with $s=\infty$ and with
 the assumption $\divv\bb\in L^1_{\mathrm{loc}}\left([0,\tau);L^1(\hh^n)\right)$ grants uniqueness, while a simple smoothing argument based on the assumption
 $[\divv\bb]^-\in L^1_{\mathrm{loc}}\left([0,\tau);L^\infty(\hh^n)\right)$ grants existence, see the above mentioned Lecture Notes for details.

\section{Counterexamples and connection with existing results}\label{sec:connection}
In this section, we emphasize the relevance of our approach by showing how the existing results in the literature cannot generally be applied to our case.
\subsection{Contact velocity fields without Euclidean regularity}
First, we show that contact velocity fields with horizontal Sobolev regularity are, \emph{generically}, not of Euclidean bounded variation, whence the renormalization results obtained in \cite{MR2096794,MR1022305} do not apply. By generically, we refer to the vocabulary of \emph{Baire's categories} (see e.g.~\cite{MR2827930}). 
We recall that, given a metric space $(M,d)$, a set $A\subseteq M$ is of the \emph{first category} if $A$ is contained in a countable union of \emph{nowhere dense} sets. In turn, a set $E\subseteq M$ is said to be nowhere dense if its closure has empty interior. The complement of a set of the first category is called \emph{generic}. If $(M,d)$ is complete and non-empty, Baire's category theorem asserts that generic sets are dense, whence in particular non-empty. 

In order to prove that vector fields with horizontal Sobolev regularity are, generically, not of locally bounded variation, we use the following basic functional analytic lemma.

\begin{lemma} \label{Baire} Let $(X,\|\cdot\|_X)$ be a normed space and let $Y\subseteq X$ be a subspace, endowed with a norm $\|\cdot\|_Y$ such that
\begin{equation}\label{eq:generic}
\sup_{y\in Y : \|y\|_X=1}\|y\|_Y=\infty.
\end{equation}
If the function $g:X\to [0,\infty]$ defined by
\begin{equation}
g(x):=
\begin{cases}
\|x\|_Y &\text{if $x\in Y$}\\
\infty &\text{otherwise}
\end{cases}
\end{equation}
is lower semicontinuous in $X$, then $X\setminus Y$ is generic in $X$.
\end{lemma}
\begin{proof} Without loss of generality, we can assume that $Y$ is dense in $X$ (since any closed proper subspace of $X$ is of first
category). It is sufficient to prove that the closed sets $A_k=\{x: g(x)\leq k\}$, whose union is $Y$, have empty interior. If for some integer $k$ this does not happen, there exist $x_0\in X$ and $\varepsilon>0$ such that $\|x-x_0\|_X<\varepsilon$ implies $g(x)\leq k$. The density of $Y$ in
$X$ grants the existence of $y_0\in Y$ such that $\|x-y_0\|_X<\varepsilon/2$ implies $g(x)\leq k$. In particular, 
$\|x-y_0\|_X<\varepsilon/2$ implies $g(x-y_0)\leq k+\|y_0\|_Y$. This last implication gives that the supremum in \eqref{eq:generic} is bounded by $(2k+2\|y_0\|_Y)/\varepsilon$, a contradiction.
\end{proof}

In the class of contact vector fields $\bb=\bb(\psi)$ parametrized by $\psi\in W^{2,s}(\hh^n)$ as in \eqref{contact_vector_field} and with the induced topology we can prove the following genericity result. The strategy of proof is to use highly oscillating test functions to provide estimates on different scales for derivatives along different directions. In particular, we will produce competitors whose oscillation in the vertical direction is predominant compared to horizontal derivatives. Accordingly, this approach grants the existence of contact vector fields with horizontal Sobolev regularity and unbounded vertical variation. In connection with Baire's category arguments, by contrast, this could be compared with the main result of \cite{MR1648524}, which roughly speaking asserts that the class of renormalizable Euclidean vector fields is generic in $L^1$. The strategy of proof of \cite{MR1648524} is to consider the class of vector fields which can be approximated in $L^1$ by Lipschitz ones at a rate faster than the inverse of their Lipschitz constant.

\begin{theorem}\label{thm_genericity1}
    The set
    \begin{equation*}
     \left \{\psi\in W^{2,s}_\hhh(\hh^n)\,:\,\,b_1(\psi),\ldots,b_{2n}(\psi)\notin BV_{\mathrm{loc}}(\rr^N)\right\}
    \end{equation*}
    is generic, and in particular dense, in $W^{2,s}_\hhh(\hh^n)$.
\end{theorem}
\begin{proof}
We fix a bounded open set $\Om\subseteq\rr^N$ and $i\in\{1,\ldots,2n\}$. We assume without loss of generality that
$0\in\Om$ and we are going to apply \Cref{Baire} with $X=W^{2,s}_\hhh(\hh^n)$ and the subspace
$$
Y\coloneqq \left\{\psi\in W^{2,s}_\hhh(\hh^n) : b_i(\psi)\in BV(\Omega)\right\}
$$
endowed with the norm $\|\psi\|_{W^{2,s}_\hhh(\hh^n)}+|Db_i(\psi)|(\Omega)$.
Notice that the standard formula of the $BV$ theory 
$$
|Db_j(\psi)|(\Omega)=\sup\left\{\int_\Omega b_i(\psi)\divv F\, dp : F\in C^1_c(\Omega;\rr^N),\,\,|F|\leq 1\right\}
$$
grants that the lower semicontinuity assumption of \Cref{Baire} is satisfied. Hence, it will be sufficient to show that there is no constant $C$ such that
\begin{equation}\label{eq:Baire1}
|D b_i(\psi)|(\Om)\leq C\|\psi\|_{W^{2,s}_\hhh(\hh^n)}\qquad\forall \psi\in W^{2,s}_\hhh(\hh^n).
\end{equation}
Let an open Euclidean ball $B\subseteq B(0,1)$ of $\rr^{2n}$ and an open interval $I\subseteq (-1,1)$, both centered at $0$, be chosen in such a way that $\overline B\times \overline I\subseteq\Om$. Fix $\varphi\in C^\infty_c(\rr^{2n})$ such that $\mathrm{supp}(\varphi)\subseteq B$, $\|\varphi\|_{L^1(B)}=1$ and $|\partial _{w_i}\varphi\|_{L^1(B)}\neq 0$. Fix $g\in C^\infty_c(\rr)$ such that $\mathrm{supp}(g)\subseteq I$ and $\|g\|_{L^1(I)}=1$. Denote points in $\rr^{2n}$ by $w$. For $\beta,\delta >0$, set
   \begin{equation*}
   \psi(w,t)=\varphi\left(\frac{w}{\beta}\right)g\left(\frac{t}{\delta}\right)\qquad\text{for every $w\in\rr^{2n},\,t \in\rr$.}
   \end{equation*}
    For $\beta$ and $\delta$ sufficiently small, $\mathrm{supp}(\psi)\subseteq B\times I$. Recall that $Z_j=\partial_{w_j}-2\JJ(w)_j\partial _t$ for every $j=1,\ldots,2n$. For every $j,k=1,\ldots,2n$, set $c_{jk}=-1$ if $j=1,\ldots,n$ and $k=j+n$, $c_{jk}=1$ if $j=n+1,\ldots,2n$ and $k=j-n$ and $c_{jk}=0$ otherwise. Then
    \begin{align*}
        Z_j\psi(w,t)&=\frac{1}{\beta}\left(\partial_{w_j}\varphi\right)\left(\frac{w}{\beta}\right)g\left(\frac{t}{\delta}\right)-\frac{2}{\delta}\JJ(w)_j\varphi\left(\frac{w}{\beta}\right)g'\left(\frac{t}{\delta}\right),\\
        Z_kZ_j\psi(w,t)&=\frac{1}{\beta^2}\left(\partial_{w_k}\partial_{w_j}\varphi\right)\left(\frac{w}{\beta}\right)g\left(\frac{t}{\delta}\right)-\frac{2}{\delta}c_{jk}\varphi\left(\frac{w}{\beta}\right)g'\left(\frac{t}{\delta}\right)-\frac{2}{\beta\delta}\JJ(w)_j\left(\partial_{w_k}\varphi\right)\left(\frac{w}{\beta}\right)g'\left(\frac{t}{\delta}\right)\\
        &\quad-\frac{2}{\beta\delta}\JJ(w)_k\left(\partial_{w_j}\varphi\right)\left(\frac{w}{\beta}\right)g'\left(\frac{t}{\delta}\right)+\frac{4}{\delta^2}\JJ(w)_j\JJ(w)_k\varphi\left(\frac{w}{\beta}\right)g''\left(\frac{t}{\delta}\right),\\
        TZ_i\psi(w,t)&=\frac{1}{\beta\delta}\left(\partial_{w_i}\varphi\right)\left(\frac{w}{\beta}\right)g'\left(\frac{t}{\delta}\right)-\frac{2}{\delta^2}\JJ(w)_i\varphi\left(\frac{w}{\beta}\right)g''\left(\frac{t}{\delta}\right).
    \end{align*}
    for every $j,k=1,\ldots,2n$ and every $(w,t)\in\Om$. Set
    \begin{equation*}
        F_1=1+\sum_{j=1}^{2n}\|\partial_{w_j}\varphi\|_{L_1(B)},\qquad F_2=F_1+\sum_{j,k=1}^{2n}\|\partial_{w_k}\partial_{w_j}\varphi\|_{L_1(B)},\qquad G_1=\|g'\|_{L_1(I)},\qquad G_2=\|g''\|_{L^1(I)}.
    \end{equation*}
    Then
    \begin{align*}
        \|\psi\|_{L^1(\Om)}&=\beta^{2n}\delta,\\
        \|Z_j\psi\|_{L^1(\Om)}&\leq \beta^{2n-1}\delta F_1+2\beta^{2n+1}G_1,\\
        \|Z_kZ_j\psi\|_{L^1(\Om)}  &\leq  \beta^{2n-2}\delta F_2+2\beta^{2n}\left(G_1+2F_1 G_1\right)+\frac{4\beta^{2n+2}}{\delta}G_2,\\
        \|TZ_i\psi\|_{L^1(\Om)}&\geq \beta^{2n-1}\|\partial _{w_i}\varphi\|_{L^1(B)}G_1-\frac{2\beta^{2n+1}}{\delta}G_2.
    \end{align*}
    In particular, if we choose $\delta=M\beta^2$, with $M= 4G_2/(\|\partial _{w_i}\varphi\|_{L^1(B)}G_1)$, in such a way that $\|TZ_i\psi\|_{L^1(\Om)}$ is
    larger than $\beta^{2n-1}\|\partial _{w_i}\varphi\|_{L^1(B)}G_1/2$, we see that all the terms in the first three lines above are $O(\beta^{2n})$. Hence, there is no constant $C$
    such that \eqref{eq:Baire1} holds.
    \end{proof}
    
\subsection{Contact velocity fields with deformation of type $(r,s)$}
In this section, we show that our results cannot be recovered by employing the metric approach of \cite{MR3265963}.
As already mentioned in the introduction, the authors of \cite{MR3265963} showed the renormalization property for suitable classes of velocity fields in a broad class of metric measure spaces, equipped with a \emph{Dirichlet form}, a \emph{carré du champ}, a \emph{Markov semigroup} and an algebra of test functions. We briefly specialize \cite{MR3265963} to our setting.
On the metric measure space $(\hh^n,d,\mathcal L^N)$, we introduce the Dirichlet form $\mathscr E:L^2(\hh^n)\to[0,\infty]$ by
\begin{equation*}
    \mathscr E(u)\coloneqq\displaystyle{\begin{cases}
    \displaystyle\int_{\hh^n} \left\langle\nabla^\hhh u,\nabla^\hhh u\right\rangle\,d p&\text{ if }u\in W^{1,2}_\hhh(\hh^n)\\
    \infty&\text{otherwise}.
    \end{cases}}
\end{equation*}
By the $L^2$-lower semicontinuity of $\mathcal E$ (see \cite{MR1437714,lic}), the latter coincides with the Cheeger energy of $(\hh^n,d,\mathcal L^N)$. The quadratic form 
$\mathscr E$  is induced by the \emph{carré du champ} $\Gamma:W^{1,2}_\hhh(\hh^n)\to L^1(\hh^n)$ defined by 
\begin{equation*}
    \Gamma(u)\coloneqq\left\langle\nabla^\hhh u,\nabla^\hhh u\right\rangle\qquad\text{for every $u\in W^{1,2}_\hhh(\hh^n)$.}
\end{equation*}
We recall that 
\begin{equation}\label{eq_formulapercollegareilaplaciani}
    \int_{\hh^n} \left\langle\nabla^\hhh u,\nabla^\hhh v\right\rangle\,d p
    =-\int _{\hh^n}v\Delta^\hhh u\,dp \qquad\text{ for every $u,v\in W^{1,2}_\hhh(\hh^n)$ such that $\Delta^\hhh v\in L^2(\hh^n)$},
\end{equation}
where the \emph{horizontal Laplacian} $\Delta^\hhh$ is defined by
\begin{equation*}
    \Delta^\hhh u\coloneqq\divv\nabla^\hhh u=\sum_{i=1}^{2n}Z_iZ_i u \qquad\text{for every $u\in C^\infty(\hh^n)$},
\end{equation*}
and then naturally extended to distributions. By \eqref{eq_formulapercollegareilaplaciani}, $\Delta^\hhh$ is the correct Laplace operator induced by $\mathscr E$. Moreover, the $L^2$-gradient flow associated with $\mathscr E$ is the classical sub-Riemannian \emph{heat semigroup} $P_\tau$ (see \cite{MR0657581}). Finally, a suitable algebra of test functions could be the class $C^1_c(\hh^n)$. The uniqueness scheme of \cite{MR3265963} relies on suitable regularization properties of the heat semigroup, as well as on density assumptions on the algebra and on regularity conditions on the velocity field. First, $P_\tau$ is required to satisfy the so-called \emph{$L^s$-$\Gamma$ inequality} for any $s\in[2,\infty)$, i.e.,
\begin{equation}\label{eq_gammalpinequality}
    \left\|\sqrt{\Gamma(P_\tau f)}\right\|_{L^s(\hh^n)}\leq \frac{c_s}{\sqrt{\tau}}\|f\|_{L^s(\hh^n)}\qquad\text{for every $f\in L^s(\hh^n)$, $\tau\in(0,1)$,}
\end{equation}
where $c_s>0$ is independent of $f$. As pointed out in \cite[Corollary 6.3]{MR3265963}, \eqref{eq_gammalpinequality} is a typical consequence of the \emph{Bakry-\'Emery condition} $\mathrm{BE}_2(K,\infty)$. Namely, we recall that $\mathrm{BE}_2(K,\infty)$ holds if there exists $K\in\rr$ such that 
\begin{equation*}
    \Gamma\left(P_\tau f\right)\leq e^{-2K\tau}P_\tau\left(\Gamma f\right)\qquad\text{ for every $f\in W^{1,2}_\hhh(\hh^n)$.}
\end{equation*}
However, $(\hh^n,d,\mathcal L^N)$ is a prototypical setting in which $\mathrm{BE}_2(K,\infty)$ does not hold for any $K$. Indeed, it is well-known that $(\hh^n,d,\mathcal L^N)$ fails to satisfy the so-called \emph{curvature-dimension condition} $\mathrm{CD(K,\infty)}$ for any $K\in\rr$ (cf.~\cite{MR4061989,MR2520783}). In turn, when the Cheeger energy is a quadratic form, $\mathrm{CD(K,\infty)}$ is equivalent to $\mathrm{BE_2(K,\infty)}$ (cf.~\cite{MR3205729, MR3298475}), whence the latter fails in $(\hh^n,d,\mathcal L^N)$ for every $N \in \mathbb{N} \setminus \{0\}$ and every $K\in\rr$. Nevertheless, it is still possible to prove the validity of \eqref{eq_gammalpinequality}.
\begin{proposition}
    For every $s\in[2,\infty)$, \eqref{eq_gammalpinequality} holds.
\end{proposition}
\begin{proof}
    Fix $s\in[2,\infty)$, $\tau\in(0,1)$ and $f\in C^\infty_c(\hh^n)$. Then, by \cite[Remark 3.3]{MR2462581},
    \begin{equation*}
      \Gamma(P_\tau f)\leq\frac{1}{\tau}\left(\frac{n+1}{2n}\right)\left[P_\tau(f^2)-\left(P_\tau f\right)^2\right]\leq \frac{1}{\tau}P_\tau(f^2).
    \end{equation*}
    In particular,
    \begin{equation}\label{eq_gammaineqsmoothproof}
        \left\|\sqrt{\Gamma(P_\tau f)}\right\|_{L^s(\hh^n)}
        \leq\frac{1}{\sqrt{\tau}}\left\|P_\tau\left(f^2\right)\right\|_{L^\frac{s}{2}(\hh^n)}\leq \frac{1}{\sqrt{\tau}}\left\|f\right\|_{L^s(\hh^n)},
    \end{equation}
    where the last inequality follows by the $L^{\frac{s}{2}}$-contractivity of $P_\tau$ (cf.~\cite{MR0657581}). Fix $f\in L^s(\hh^n)$, and let $\{f_h\}_{h\in\mathbb N}\subseteq C^\infty_c(\hh^n)$ be such that $f_h\to f$ strongly in $L^s(\hh^n)$. By linearity and contractivity, $P_\tau f_h\to P_\tau f$ strongly in $L^s(\hh^n)$, whence, by the strong $L^s$-lower-semicontinuity (cf.~\cite{MR1437714,lic}) of $ u\mapsto \|\sqrt{\Gamma\left(\nabla^\hhh u\right)}\|_{L^s(\hh^n)}$, 
we conclude that 
\begin{equation*}
    \left\|\sqrt{\Gamma(P_\tau f)}\right\|_{L^s(\hh^n)}\leq\liminf_{h\to\infty}\left\|\sqrt{\Gamma(P_\tau f_h)}\right\|_{L^s(\hh^n)}\overset{\eqref{eq_gammaineqsmoothproof}}{\leq}\frac{1}{\sqrt{\tau}}\liminf_{h\to\infty}\left\|f_h\right\|_{L^s(\hh^n)}=\frac{1}{\sqrt{\tau}}\left\|f\right\|_{L^s(\hh^n)}.
\end{equation*}
\end{proof}
Coming to the density properties of the algebra of test functions $C^1_c(\hh^n)$, it is easy to check that the latter satisfies the basic assumptions \cite[(2-16)]{MR3265963} and \cite[(2-17)]{MR3265963}. In addition (cf.~\cite[(4-3)]{MR3265963}), the algebra is required to contain a sequence $\{f_h\}_{h\in\mathbb N}$ such that 
\begin{equation}\label{eq_trepropalgebra}
    0\leq f_h\leq 1\text{ for every $h\in\mathbb N$,}\qquad f_h\nearrow 1\quad\text{ a.e.~on $\hh^n$,}\qquad \sqrt{\Gamma(f_h)}\to 0\text{ weakly$^\star$ in $L^\infty(\hh^n)$.}
\end{equation}
A sequence satisfying \eqref{eq_trepropalgebra} can be easily constructed by choosing a sequence of smooth cut-off functions between $B(0,h+1)$ and $B(0,h+2)$. Finally, we turn to the regularity assumptions on the velocity field. For the sake of simplicity, we focus on the autonomous case. From now on, we fix $s\in(1,\infty]$ and a vector field $\bb\in\VfH{n}$ with $b_i\in W^{1,s}_{\hhh}(\hh^n)$, $i=1,\ldots,2n$ and $Tb_N\in L^s(\hh^n)$,
so that $\divv\bb\in L^s(\hh^n)$. Following \cite{MR3265963}, $\bb$ induces a \emph{derivation} acting on a test function $f\in C^\infty_c(\hh^n)$ as 
\begin{equation*}
    df(\bb)\coloneqq\langle\bb,\nabla f\rangle.
\end{equation*}
Again, according to \cite{MR3265963}, the \emph{deformation} $\dsym\bb$ is defined for every $f,g\in C^\infty_c(\hh^n)$ by
\begin{equation}\label{definizioonedsym}
    \int_{\hh^n}\dsym\bb(f,g)\,dp\coloneqq
    -\frac{1}{2}\int_{\hh^n}\left[df(\bb)\Delta^\hhh g+dg(\bb)\Delta^\hhh f-\divv\bb\,\left\langle\nabla^\hhh f,\nabla^\hhh g\right\rangle\right]\,dp.
\end{equation}
Clearly, the above definition extends to any couple $f,g$ for which the right hand side is meaningful.
For what concerns $\bb$, the approach of \cite{MR3265963} relies on the following two assumptions.
\begin{itemize}
    \item $\bb$ needs to satisfy an a.e.~upper bound of the form
    \begin{equation}\label{upperboundvf}
        |df(\bb)|\leq |\bb|\left|\nabla^\hhh f\right|
    \end{equation}
    for every $f$ belonging to the algebra of test functions. 
    \item $\bb$ needs to have a \emph{deformation of type $(r,s)$.} Namely, there exists $c\geq 0$ such that 
\begin{equation}\label{eq_defoftypedef}
      \int_{\hh^n}\dsym\bb(f,g)\,dp\leq c\left\|\sqrt{\left(\nabla^\hhh f\right)}\right\|_{L^r(\hh^n)}\left\|\sqrt{\left(\nabla^\hhh g\right)}\right\|_{L^s(\hh^n)}
\end{equation}
for every $f,g$ for which the above expression is meaningful.
\end{itemize}
 On the one hand, it is clear that \eqref{upperboundvf} rules out non-horizontal vector fields. On the other hand, \eqref{eq_defoftypedef} rules out vector fields which do not possess a contact structure. Since a non-trivial contact vector field cannot be horizontal, the above two restrictions prevent the approach of \cite{MR3265963} from being applicable to our case. To justify the validity of the second restriction (independently of the first one, hence even considering a non-horizontal field $\bb$), we assume for the sake of simplicity that all components $b_i$ belong to $C^\infty_c(\hh^n)$, and we make \eqref{definizioonedsym} more explicit. Precisely, if $f,g\in C^\infty_c(\hh^n)$,
 \begin{equation}\label{eq_defofinel}
\begin{split}
      \int_{\hh^n}\dsym\bb(f,g)\,dp&=\int_{\hh^n}\sum_{i,j=1}^{2n}\left(\frac{Z_i b_j+Z_jb_i}{2}\right)Z_i fZ_j g\,dp\\
      &\quad+\frac{1}{2}\int_{\hh^n}\left\langle \nabla^\hhh b_N+4\JJ(\bb),Tg\,\nabla^\hhh f+Tf\,\nabla^\hhh g\right\rangle\,dp.
\end{split}
\end{equation}
 The validity of \eqref{eq_defofinel} follows by applying the divergence theorem to the first term on the right hand side and by exploiting the commutation relations \eqref{eq_commutation}. Indeed, notice that 
 \begin{equation*}
        \begin{split}
            \frac{1}{2}&\int_{\hh^n}\sum_{i,j=1}^{2n}Z_i b_j\left(Z_i fZ_j g+Z_j fZ_i g\right)\,dp =\frac{1}{2}\sum_{i,j=1}^{2n}\underbrace{\int_{\hh^n}Z_i\left(b_j\left(Z_i fZ_j g+Z_j fZ_i g\right)\right)\,dp}_{=0}\\
            &\quad-\frac{1}{2}\int_{\hh^n}\sum_{i,j=1}^{2n}\Bigl(\underbrace{b_jZ_iZ_i fZ_j g}_\mathrm{I}+\underbrace{b_jZ_i fZ_iZ_j g}_\mathrm{II}+\underbrace{b_jZ_iZ_j fZ_i g}_\mathrm{III}+\underbrace{b_jZ_j fZ_iZ_i g}_\mathrm{IV}\Bigr)\,dp\\
            &=-\frac{1}{2}\int_{\hh^n}\left[df(\bb)\Delta^\hhh g+dg(\bb)\Delta^\hhh f\right]\,dp+\frac{1}{2}\int_{\hh^n}\left(\Delta^\hhh f\, b_N\,Tg+\Delta^\hhh g\, b_N\,Tf\right)\,dp\\
            &\quad-\frac{1}{2}\int_{\hh^n}\sum_{i,j=1}^{2n}\left(b_jZ_i fZ_jZ_i g+b_jZ_jZ_i fZ_i g\right)\,dp+\frac{1}{2}\int_{\hh^n}\sum_{i,j=1}^{2n}\left(b_jZ_i f[Z_j,Z_i] g+b_j[Z_j,Z_i]fZ_i g\right)\,dp.
        \end{split}
    \end{equation*}
   In the last equality, we exploited $\mathrm{I}$ and $\mathrm{IV}$ to make the first two terms on the right hand side of \eqref{definizioonedsym} explicit, and rewrote $\mathrm{II}$ and $\mathrm{III}$ highlighting commutators. Applying again the divergence theorem,
    \begin{equation*}
        \begin{split}
            -\frac{1}{2}\int_{\hh^n}\sum_{i,j=1}^{2n}&\left(b_jZ_i fZ_jZ_i g+b_jZ_jZ_i fZ_i g\right)\,dp= -\frac{1}{2}\sum_{i,j=1}^{2n}\underbrace{\int_{\hh^n}Z_j\left(b_jZ_i fZ_i g\right)\,dp}_{=0}\\
            &\quad+\frac{1}{2}\int_{\hh^n}\sum_{i,j=1}^{2n}Z_jb_jZ_i fZ_i g\,dp+\frac{1}{2}\int_{\hh^n}\sum_{i,j=1}^{2n}b_jZ_jZ_i fZ_i g\,dp-\frac{1}{2}\int_{\hh^n}\sum_{i,j=1}^{2n}b_jZ_jZ_i fZ_i g\,dp\\
            &=\frac{1}{2}\int_{\hh^n}\divv\bb\,\left\langle \nabla^\hhh f,\nabla^\hhh g\right\rangle\,dp-\frac{1}{2}\int_{\hh^n}Tb_N\,\left\langle \nabla^\hhh f,\nabla^\hhh g\right\rangle\,dp.
        \end{split}
    \end{equation*}
Moreover, by \eqref{eq_commutation},
    \begin{equation*}
        \begin{split}
            \frac{1}{2}\int_{\hh^n}\sum_{i,j=1}^{2n}&b_jZ_i f[Z_j,Z_i] g\,dp+\frac{1}{2}\int_{\hh^n}\sum_{i,j=1}^{2n}b_j[Z_j,Z_i]fZ_i g\,dp\\
            &=\frac{1}{2}\int_{\hh^n}\sum_{j=1}^n\left(b_jY_j f[X_j,Y_j] g+b_{n+j}X_j f[Y_j,X_j] g+b_jY_j g[X_j,Y_j] f+b_{n+j}X_j g[Y_j,X_j] f\right)\,dp\\
            &=\frac{1}{2}\int_{\hh^n}\sum_{j=1}^n\left[4b_j\left(-Y_j f\right)\,T g+4b_{n+j}X_j f\,T g+4b_j\left(-Y_j g\right)\,T f+4b_{n+j}X_j g\,T f\right]\,dp\\
            &=\frac{1}{2}\int_{\hh^n}\left[\left\langle 4\bb,\JJ\left(\nabla^\hhh f\right)\right\rangle\, Tg+\left\langle 4\bb,\JJ\left(\nabla^\hhh g\right)\right\rangle\, Tf\right]\,dp.
        \end{split}
    \end{equation*}
Therefore, by the above computations,
    \begin{equation*}
        \begin{split}
            \int_{\hh^n}&\dsym\bb(f,g)\,dp=  \int_{\hh^n}\sum_{i,j=1}^{2n}\left(\frac{Z_i b_j+Z_jb_i}{2}\right)Z_i fZ_j g\,dp+\frac{1}{2}\int_{\hh^n}\left\langle 4\JJ(\bb),Tg\,\nabla^\hhh f+Tf\,\nabla^\hhh g\right\rangle\,dp\\
            &\quad-\frac{1}{2}\int_{\hh^n}\left(\Delta^\hhh f\, b_N\,Tg+\Delta^\hhh g\, b_N\,Tf\right)\,dp+\frac{1}{2}\int_{\hh^n}Tb_N\,\left\langle \nabla^\hhh f,\nabla^\hhh g\right\rangle\,dp\\
            &=  \int_{\hh^n}\sum_{i,j=1}^{2n}\left(\frac{Z_i b_j+Z_jb_i}{2}\right)Z_i fZ_j g\,dp+\frac{1}{2}\int_{\hh^n}\left\langle \nabla^\hhh b_N+4\JJ(\bb),Tg\,\nabla^\hhh f+Tf\,\nabla^\hhh g\right\rangle\,dp\\
            &\quad-\frac{1}{2}\int_{\hh^n}\left[\divv\left(b_N\nabla^\hhh f\right)\,Tg+\divv\left(b_N\nabla^\hhh g\right)\,Tf\right]
            +\frac{1}{2}\int_{\hh^n}Tb_N\,\left\langle \nabla^\hhh f,\nabla^\hhh g\right\rangle\,dp,\\
        \end{split}
    \end{equation*}
 where in the last equality we just added and subtracted the quantity
 \begin{equation*}
     \frac{1}{2}\int_{\hh^n}\left\langle \nabla^\hhh b_N,Tg\,\nabla^\hhh f+Tf\,\nabla^\hhh g\right\rangle\,dp.
 \end{equation*}
 Applying once more the divergence theorem, and since $T$ commutes with $Z_1,\ldots,Z_{2n}$, we conclude that
\begin{equation*}
    \begin{split}
        -\frac{1}{2}\int_{\hh^n}\left[\divv\left (b_N\,\nabla^\hhh f\right)\, Tg+\divv\left (b_N\,\nabla^\hhh g\right)\, Tf\right]\,dp&=\frac{1}{2}\int_{\hh^n}b_N\left(\left\langle\nabla^\hhh f,\nabla^\hhh Tg\right\rangle+\left\langle\nabla^\hhh Tf,\nabla^\hhh g\right\rangle\right)\,dp\\
        &=\frac{1}{2}\int_{\hh^n}b_N\,T\left\langle\nabla^\hhh f,\nabla^\hhh g\right\rangle\,dp\\
        &=-\frac{1}{2}\int_{\hh^n}T b_N\,\left\langle\nabla^\hhh f,\nabla^\hhh g\right\rangle\,dp, 
    \end{split}
\end{equation*}
whence \eqref{eq_defofinel} follows.
Due to the presence of the vertical derivatives $Tf$ and $Tg$ on the right hand side of \eqref{eq_defofinel}, we conclude that $\bb$ has deformation of type $(r,s)$ provided that it is a contact vector field. Indeed, in this case the second integral on the right hand side of \eqref{eq_defofinel} vanishes by \eqref{extra_J}, and \eqref{eq_defoftypedef} follows by choosing $r=s=2(s')'$ and $c=\sum_{i,j=1}^{2n}\|Z_i b_j\|_{L^{s'}(\rr^N)}$. 

\appendix
\section{Chain rule for weak derivatives along vector fields}
The following chain rule for weak derivatives along vector fields is surely well-known. Being typically stated for Sobolev functions (cf.~\cite{MR1158660,MR3726909}), we include a proof for the sake of completeness.
\begin{proposition}
    Let $m\in\mathbb{N}\setminus\{0\}$. Let $A\subseteq\rr^m$ be open. Let $X$ be a locally Lipschitz continuous vector field over $A$. 
    Let $u\in L^1_{\mathrm{loc}}(A)$ be such that $Xu\in L^1_{\mathrm{loc}}(A)$ and let $\beta\in C^1(\rr)$ be such that $\beta'\in L^\infty(\rr)$. Then $\beta(u)\in L^1_{\mathrm{loc}}(A)$, $X(\beta(u))\in L^1_{\mathrm{loc}}(A)$ and 
    \begin{equation}\label{chainrulesobolev}
        X(\beta(u))=\beta'(u)Xu\qquad\text{a.e.~on $A$.}
    \end{equation}
    Moreover, if $u\in L^\infty_{\mathrm{loc}}(A)$, the above facts hold for every $\beta\in C^1(\rr)$. 
\end{proposition}
\begin{proof}
    Set $C=\|\beta'\|_{L^\infty(\rr)}$. Fix an open set $B\Subset A$. Then
    \begin{equation}\label{tesi_chain_rule_in_proof}
        \int_B|\beta(u)|\,dp \leq C\int_B|u|\,dp+\beta(0)|B|<\infty,
    \end{equation}
    whence $\beta(u)\in L^1_{\mathrm{loc}}(A)$. Notice that $\beta'(u)Xu\in L^1_{\mathrm{loc}}(A)$. To conclude, it suffices to show that 
    \begin{equation*}
       \int_A\left(\beta'(u)Xu\right)\varphi\,dp=-\int_A\divv X\,\beta(u)\varphi\,dp -\int_A\beta(u)X\varphi\,dp\qquad\text{for every $\varphi\in C^\infty_c(A)$.}
    \end{equation*}
    Fix $\varphi\in C^\infty_c(A)$. Let $B\subseteq A$ be open and such that $\mathrm{supp}(\varphi)\Subset B\Subset A$. Then $X$ is Lipschitz continuous on $B$, and moreover $u,Xu\in L^1(B)$. By \cite[Theorem 1.2.3]{MR1437714}, there exists a sequence $\{u_h\}_{h\in \mathbb N}\subseteq C^\infty(B)$ such that $u_h, X u_h\in L^1(B)$ for any $h\in\mathbb N$, $u_h\to u$, $Xu_h\to Xu$ strongly in $L^1(B)$, $u_h\to u$ a.e.~on $B$ and $|Xu_h|\leq g$ a.e.~on $B$ for some $g\in L^1(B)$. Since every $u_h$ is smooth, $\beta'(u_h)Xu_h=X(\beta(u_h))$ on $B$ for every $h\in\mathbb N$. In particular, for every $h\in\mathbb N$,
\begin{equation}\label{coselisceinterproofderideb}
    \int_A\beta'(u_h)Xu_h\varphi\,dp
    = -\int_A\divv X\,\beta(u_h)\varphi\,dp-\int_A\beta(u_h)X\varphi\,dp.
\end{equation}
Notice that, since $\beta'$ is continuous,  $\beta'(u_h)\to\beta'(u)$ a.e.~on $B$, and moreover, $|\beta'(u_h)-\beta'(u)|\leq 2C$. Therefore, by the dominated convergence theorem,
\begin{equation*}
    \begin{split}
        \int_A|\beta'(u_h)Xu_h\varphi-\beta'(u)Xu\varphi|\,dp&\leq  \|\varphi\|_{L^\infty(B)}\int_B|\beta'(u)||Xu_h-Xu|\,dp+ \int_B|\beta'(u_h)-\beta'(u)||Xu_h||\varphi|\,dp\\
        &\leq  C\|\varphi\|_{L^\infty(B)}\int_B|Xu_h-Xu|\,dp+\int_B|\beta'(u_h)-\beta'(u)||Xu_h||\varphi|\,dp\to 0.
    \end{split}
\end{equation*}
Furthermore,
\begin{equation*}
    \int_A|\divv X\,\beta(u_h)\varphi-\divv X\,\beta(u)\varphi|\,dp\leq C\|\varphi\|_{L^\infty(B)}\|\divv X\|_{L^\infty(B)}\int_{B}|u_h-u|\,dp\to 0.
\end{equation*}
Finally, 
\begin{equation*}
    \int_A|\beta(u_h)X\varphi-\beta(u)X\varphi|\,dp\leq C\|X\varphi\|_{L^\infty(B)}\int_{B}|u_h-u|\,dp\to 0.
\end{equation*}
Hence, \eqref{tesi_chain_rule_in_proof} follows by letting $h\to\infty$ in \eqref{coselisceinterproofderideb}. In conclusion, when $u\in L^\infty_{\mathrm{loc}}(A)$, the above arguments apply for every $\beta\in C^1(\rr)$.
\end{proof}
\bibliographystyle{abbrv}

\end{document}